\documentclass[11pt]{amsart}
\usepackage{amssymb}
\usepackage{amsfonts}
\usepackage{amsmath}
\usepackage{graphicx}
\usepackage{xcolor}
\usepackage{amsmath}
\usepackage{mathrsfs}
\usepackage{stmaryrd}
\usepackage{epsfig,color}
\usepackage{blindtext}
\usepackage{enumerate}
\usepackage{hyperref}
\usepackage{url}
\usepackage{bbm}
\usepackage{nicefrac,mathtools}
\usepackage{bm}   
\usepackage{tabularx}
\DeclareGraphicsExtensions{.pdf,.jpeg,.png}
\usepackage{epstopdf}
\usepackage{cancel} %for editing
\usepackage[normalem]{ulem} %for editing
\usepackage{verbatim} %for comment
\usepackage{scalerel}
\usepackage{enumitem}
\usepackage{tikz-cd}%commutative diagram

\usepackage{color}
\usepackage[msc-links, lite]{amsrefs}
\usepackage{geometry}
\usepackage{stackengine,scalerel}

\newcommand\ttimes{\mathbin{\ThisStyle{\ensurestackMath{%
  \stackengine{-1\LMpt}{\SavedStyle\times}
  {\SavedStyle_{\hstretch{.9}{\mkern1mu\sim}}}{O}{c}{F}{T}{S}}}}}

\newtheorem{theorem}{Theorem}[section]

\newtheorem{problem}{Problem}

\newtheorem{proposition}[theorem]{Proposition}
\newtheorem{lemma}[theorem]{Lemma}

\newtheorem{claim}[]{Claim}

\theoremstyle{definition}
\newtheorem{definition}[theorem]{Definition}
\theoremstyle{remark}
\newtheorem{remark}[theorem]{Remark}

\numberwithin{equation}{section}

\newcommand{\dv}{\mathrm{div}}

\newcommand{\mf}{\mathbf}
\newcommand{\mb}{\mathbb}
\newcommand{\mc}{\mathcal}
\newcommand{\ms}{\mathscr}
\newcommand{\mk}{\mathfrak}
\newcommand{\mr}{\mathrm}
\newcommand{\oli}{\overline}

\newcommand{\wti}{\widetilde}
\newcommand{\res}{\scaleobj{1.75}{\llcorner}}

\newcommand{\Z}{\mathcal Z}
\newcommand{\A}{\mathcal A}
\newcommand{\C}{\mathcal C}

\newcommand{\Vol}{\mathrm{Vol}}
\newcommand{\Area}{\mathrm{Area}}
\newcommand{\Id}{\mathrm{Id}}

\newcommand{\rom}[1]{\expandafter\romannumeral #1}
\newcommand{\Rom}[1]{\uppercase\expandafter{\romannumeral #1}}

\DeclareMathOperator{\Ric}{Ric}

\DeclareMathOperator{\Diff}{Diff}
\DeclareMathOperator{\spt}{spt}
\DeclareMathOperator{\interior}{int}
\DeclareMathOperator{\closure}{Clos}

\DeclareMathOperator{\VC}{\mc V\C}

\title[Existence of three free boundary minimal disks]{Existence of free boundary minimal disks in convex regions}
\author{Lorenzo Sarnataro}
\address{Department of Mathematics, University of Toronto, 40 St George Street, Toronto,
ON M5S 2E4, Canada}
\email{lorenzo.sarnataro@utoronto.ca}

\author{Douglas Stryker}
\address{Department of Mathematics, Stanford University, Building 380, Stanford, CA 94305, USA}
\email{dstryker@stanford.edu}

\author{Zhichao Wang}
\address{Shanghai Center for Mathematical Science, 2005 Songhu Road, Fudan University, Shanghai, 200438, China}
\email{zhichao@fudan.edu.cn}

\author{Xin Zhou}
\address{Department of Mathematics, 531 Malott Hall, Cornell University, Ithaca, NY 14853, USA}
\email{xinzhou@cornell.edu}

\begin{document}

\begin{abstract}
We show that any three-ball with mean convex boundary contains an embedded free boundary minimal disk. Moreover, when the three-ball is a strictly convex domain with nonnegative Ricci curvature (for instance, a compact convex domain in Euclidean three-space), we prove the existence of at least three embedded free boundary minimal disks. Our approach is based on a multiplicity-one theorem for the free boundary Simon–Smith min–max theory.
\end{abstract}

\maketitle

\section{Introduction}

The existence of minimal surfaces with prescribed topology is a central theme in geometric analysis. In his celebrated 1982 problem section \cite{Yau82}, S. T. Yau posed the following question:

\begin{problem}[\cite{Yau82}*{Problem 89}]\label{problem:sphere}
Prove that there are four distinct embedded minimal spheres in any manifold diffeomorphic to $S^3$.
\end{problem}

Around the time Yau formulated this problem, Simon--Smith \cite{Smith82} proved the existence of at least one embedded minimal sphere in $S^3$ endowed with an arbitrary Riemannian metric. Their proof utilized a variant of the min-max theory for minimal hypersurfaces pioneered by Almgren \cites{Alm62,Alm65} and Pitts \cite{Pi} (see also Schoen-Simon \cite{SS} and Colding-De Lellis \cite{Colding-DeLellis03}). Subsequently, White \cite{Whi91} employed degree-theoretic methods to establish the existence of at least two embedded minimal spheres when the metric has positive Ricci curvature, and at least four when the metric is sufficiently close to the round metric. More recently, Haslhofer-Ketover \cite{HK19} proved the existence of at least two embedded minimal spheres for bumpy metrics by combining the Simon--Smith min-max theory with mean curvature flow, relying crucially on the catenoid estimate introduced by Ketover-Marques-Neves \cite{KMN16}.

We also note that branched immersed minimal spheres were obtained earlier by Sacks-Uhlenbeck \cite{Sacks-Uhlenbeck81} via min-max theory for harmonic maps (see also Colding-Minicozzi \cite{Colding-Minicozzi08b}). Recently, the last two authors \cite{Wang-Zhou23} of the present paper proved a multiplicity one theorem for the Simon--Smith min--max theory, resolving Yau's conjecture for $3$-spheres with positive Ricci curvature and for bumpy metrics.

Motivated by Yau's conjecture for $S^3$, it is natural to consider the following analogue for three-balls:

\begin{problem}\label{problem:disk}
Prove that there exist three distinct embedded free boundary minimal disks in any manifold diffeomorphic to $B^3$.
\end{problem}

Adapting the Simon--Smith min--max theory to the free boundary setting, Gr\"uter-Jost \cite{GrJo86} (see also Li \cite{Li-CPAM}) established the existence of a free boundary minimal disk in convex domains in $\mathbb{R}^3$. For immersed solutions, see the works of Struwe \cite{Str84}, Fraser \cite{Fra00}, Sun--Lin--Zhou \cite{SLZ20}, and Laurain-Petrides \cite{LP19}. Recently, Haslhofer-Ketover \cite{HK25} proved the existence of a second free boundary minimal disk in compact three-manifolds with convex boundary and nonnegative Ricci curvature. Furthermore, for generic metrics, they demonstrated the existence of three free boundary minimal disks in strictly convex three-manifolds with nonnegative Ricci curvature. We emphasize, however, that Problem \ref{problem:disk} remains open even for strictly convex domains in $\mathbb{R}^3$.

In this manuscript, we investigate the existence of free boundary minimal disks in compact three-manifolds with strictly mean convex boundary. Specifically, we study this problem under the assumption that the manifold admits no degenerate stable minimal spheres and no degenerate stable free boundary minimal disks (a condition satisfied, for instance, by strictly convex domains in $\mathbb{R}^3$).

\begin{theorem}\label{thm: main}
Let $(M^3,\partial M,g)$ be a Riemannian three-ball with strictly mean convex boundary such that there are no degenerate-and-stable minimal spheres or degenerate-and-stable free boundary minimal disks. Then there exist at least three distinct embedded submanifolds which are either minimal spheres or free boundary minimal disks. 

In particular, if $(M^3,g)$ is a strictly convex domain with nonnegative Ricci curvature (e.g., a compact convex ball in Euclidean three-space), then it admits at least three distinct embedded free boundary minimal disks.   
\end{theorem}

Recall that in the free boundary Simon--Smith min--max theory, the resulting minimal surface may be a closed embedded sphere, even within a strictly convex three-ball. As a consequence of Theorem \ref{thm: main}, we obtain the following structural dichotomy:

\begin{theorem}\label{thm:always one fbmd}
Let $(M^3,\partial M,g)$ be a Riemannian three-ball with strictly mean convex boundary. Then there exists an embedded free boundary minimal disk in $M$. Moreover, one of the following alternatives holds:
\begin{enumerate}
    \item $M$ contains no closed immersed minimal surface; or
    \item $M$ contains an embedded stable minimal sphere.
\end{enumerate}
\end{theorem}

The primary analytical tool driving these results is a multiplicity one theorem for the Simon--Smith min--max theory in compact manifolds with boundary:

\begin{theorem}
\label{thm:classical multiplicity one with boundary}
Let $(M^3,\partial M,g)$ be a compact Riemannian three-manifold with smooth boundary. Suppose that $\Pi$ is a homotopy class of $(X, Z)$-sweepouts by $\Sigma_0$ with 
\[ \mf L(\Pi)>\sup_{x\in Z}\mc H^2\big(\Phi_0(x)\big)=0. \]
Then there exists a pairwise disjoint collection of connected, almost properly embedded, minimal surfaces $\{\Gamma_j\}_{j=1}^N$ with (possibly empty) free boundary, and positive integers $\{m_j\}_{j=1}^N$ so that 
\[  \mf L(\Pi)= \sum_{j=1}^N m_j\mc H^2 (\Gamma_j), \]	
and
\begin{enumerate}
    \item if $\Gamma_j$ is unstable and two-sided, then $m_j=1$;
    \item if $\Gamma_j$ is one-sided, then the connected double cover of $\Gamma_j$ is stable.
\end{enumerate}
Furthermore,  by Franz-Schulz \cite{Franz-Schulz-boundary-components}*{Theorem 1.8}, if $M$ is orientable and has mean convex boundary, then
\[
    \beta_1(\Gamma)\leq \beta_1(\Sigma_0),\quad   \sum_{i\in I_O}m_i\mk g(\Gamma_i)+\frac{1}{2}\sum_{i\in I_U}m_i(\mk g(\Gamma_i)-1)\leq \mk g(\Sigma_0),
\]
where $\beta_1(\Gamma)$ is the first Betti number of $\Gamma$, $\mk g(\Gamma)$ is the genus of $\Gamma$, and $I_O$ (resp. $I_U$) is the collection of $i$ such that $\Gamma_i$ is orientable (resp., non-orientable).
\end{theorem}

To produce free boundary minimal surfaces via the Simon--Smith min-max theory, a natural approach is to consider isotopies of the ambient manifold in the tightening process and to establish the almost minimizing property. While this approach works well for (mean) convex boundaries, it encounters intrinsic obstacles in more general settings. In \cite{Li-CPAM}, M. Li introduced outer variations and developed a variant of the Simon--Smith min--max theory to produce free boundary minimal surfaces; later on, the full regularity was established by Li and the last-named author in the general settings \cite{LZ16} (although the Li-Zhou paper is in the Almgren-Pitts setting, the specific regularity theory goes through unchanged in the Simon-Smith setting). Our multiplicity one theorem applies to both settings (mean convex cases and general cases).

\subsection*{Overview of strategy}

In general, Simon--Smith min--max for sweepouts by disks may produce closed minimal spheres in the interior. To ensure the existence of a free boundary minimal disk in any strictly mean convex three-ball $(M, g)$, we argue as follows. First, we run the mean curvature flow starting from $\partial M$. By a result of White \cite{Whi00}, we find a compact subset $\Omega \subset M$ containing $\partial M$ whose interior admits a (possibly singular) foliation by strictly mean convex hypersurfaces, and so that $\partial \Omega \cap \mathrm{Int}(M)$ is either empty or consists of stable embedded minimal spheres with contracting neighborhoods. We then apply the cylindrical manifold construction of Song \cite{Song18} at these stable spheres and show that the min--max width for sweepouts by disks in the cylindrical manifold is bounded. As a consequence, we produce $\Sigma$ either a free boundary minimal disk in $\Omega \setminus (\partial \Omega \cap \mathrm{Int}(M))$ or a minimal sphere in $\mathrm{Int}(\Omega)$. By the maximum principle applied to the strictly mean convex foliation, $\Sigma$ cannot be closed, so $\Sigma$ must be a free boundary minimal disk.

To produce more free boundary minimal disks (or minimal spheres), we argue similarly to \cite{Wang-Zhou23}*{Theorem A}, which in our setting requires a multiplicity one result for free boundary Simon--Smith min--max. Our proof of multiplicity one follows the strategy of \cite{Wang-Zhou23}*{Theorem B}. To briefly outline the essential idea, we first solve a sequence of min--max problems for the prescribed mean curvature functional for carefully chosen prescription functions tending to zero. Under the assumption of convergence with higher multiplicity in the limit, we extract a nonnegative supersolution to an equation involving the Jacobi operator, which forces the limit to be stable. The key new observation of this article is that the prescription functions can be chosen to have support strictly in the complement of a neighborhood of the ambient manifold's boundary. Consequently, the interior regularity theory for nonzero prescribed mean curvature developed by the authors \cites{SS23, Wang-Zhou-C11-2023, Wang-Zhou23} can be applied directly, circumventing the intensive technical task of developing the corresponding regularity theory at the boundary.

\subsection*{Organization}

The paper is organized as follows. In Section \ref{sec:pre}, we introduce the necessary notation and review the regularity theory for free $h$-boundary surfaces. In Section \ref{sec:multi 1}, we prove the multiplicity one theorem for the free boundary Simon--Smith min-max theory. Section \ref{sec:disks} is devoted to establishing the existence of three free boundary minimal disks in mean convex three-manifolds with nonnegative Ricci curvature, alongside the dichotomy result for the general case. Finally, in Appendix \ref{appen:barries}, we recall the construction of free boundary minimal surfaces in the presence of a stable free boundary minimal surface serving as a barrier, and in Appendix \ref{appen:LS inequality}, we provide a detailed proof of the Lusternik--Schnirelmann inequality in our setting.

\subsection*{Acknowledgments}
D.S. was supported by the NSF grant DMS-2503279.
Z.W. is supported by Tianyuan Mathematics Frontier Key Special Program (NSFC Grant No. 12526203).
X.Z. acknowledges the support by NSF grants DMS-1945178, DMS-2506717 and a grant from the Simons Foundation.

\section{Preliminaries}\label{sec:pre}

\subsection{Notations}
We will not restrict the ambient manifold to a three-ball until the final section. 
\begin{itemize}
    \item $M^3$ denotes a compact, oriented, three-dimensional Riemannian manifold with smooth boundary $\partial M$, and $U\subseteq M$ denotes a relatively open subset ($U$ may be equal to $M$).
    \item $\mr{An}(p; s, r)$ denotes an open annulus given by $B(p, r)\setminus \closure(B(p, s))$ for $p\in M$ and $0<s<r$. 
    \item $h\in C^{\infty}(M)$ denotes a smooth function on $M$, which will play the role as a prescribing function mean curvature.
    \item $\mathcal{C}(M)$ or $\mathcal{C}(U)$ denotes the space of sets $\Omega\subset M$ or $\Omega\subset U\subset M$ with finite perimeter (Caccioppoli sets); see \cite{Si}*{\S 14}.
    \item $\mc V(M)$ or $\mc V(U)$ denotes the space of $2$-varifolds in $M$ or $U$.
    \item $\mathfrak X(U)$ denotes the space of compactly supported smooth vector fields in $U$.
    \item $\Diff_0(M)$ denotes the connected component of the diffeomorphism group of $M$ containing the identity, and $\mk{Is}(U)$ denotes the set of isotopies of $M$ supported in $U$.
    \item A collection of connected $C^1$-embedded surfaces $\{\Gamma^i\}_{i=1}^\ell\subset U$ with $\partial\Gamma^i\cap U=\emptyset$ is said to be \emph{ordered}, denoted by 
    \[\Gamma^1\leq \cdots\leq \Gamma^\ell,\] 
    if for each $i$, $\Gamma^i$ separates $U$ into two connected components $U^i_+$ and $U^i_-$ (i.e., $U\setminus \Gamma^i = U^i_+\sqcup U^i_-$), such that $\Gamma^j\subset \closure(U^i_-)$ for $j=1, \cdots, i-1$, and $\Gamma^j\subset \closure(U^i_+)$ for $j=i+1, \cdots, \ell$.
\end{itemize}

\subsection{Free boundary min-max $h$-surfaces}
We recall the free boundary min--max theory for $h$-surfaces in dimension three established by the last two authors in \cite{Wang-Zhou23}*{Section 8.1}. For the purposes of this paper, we adapt this theory by employing isotopies of a mixed nature: they preserve the mean convex boundary components while allowing points to be pushed out through the remaining boundary components.

Let $(M, \partial M, g)$ be a compact Riemannian three-manifold with smooth boundary. Denote by
\begin{itemize}
    \item $\partial_{mc} M$ the union of boundary components that are mean convex with respect to the outward unit normal of $\partial M$, and
    \item $\partial_{other} M$ the union of the remaining boundary components.
\end{itemize}
%Note that $\partial_{other} M$ may also contain some mean convex subdomains. 
We write
\[
\partial M = \partial_{mc} M \sqcup \partial_{other} M.
\]

Assume that $(M^3, \partial M, g)$ is isometrically embedded as a submanifold of another compact Riemannian three-manifold with boundary $(\widetilde M^3, \partial \widetilde M, g)$, such that
\[
\partial \widetilde M = \partial_{mc} M.
\]
In other words, $\widetilde M$ is obtained from $M$ by filling in the boundary components $\partial_{other} M$.

Let $h: M \to \mathbb{R}$ be a fixed smooth function satisfying
\[
h = 0 \quad \text{in a neighborhood of } \partial M.
\]
Fix a compact surface $\Sigma_0$ of genus $\mathfrak{g}_0$ with $k$ boundary components. We say that a smooth map 
\[
f: (\Sigma_0, \partial \Sigma_0) \to (\widetilde M, \partial \widetilde M)
\]
is a \emph{proper embedding} if $f|_{\interior (\Sigma_0)}$ and $f|_{\partial \Sigma_0}$ are embeddings into $\interior (\widetilde M)$ and $\partial \widetilde M$, respectively. We will often identify an embedding with its image and write $(\Sigma, \partial \Sigma) = \big(f(\Sigma_0), f(\partial \Sigma_0)\big)$.

Similarly as in \cite{Wang-Zhou23}*{Equation (2.1)}, we define
\[ \ms E = \left\{(\Sigma, \Omega): \text{ $\Sigma$ is a proper separating  embedding of $\Sigma_0$ which bounds $\Omega$ in $\widetilde M$}\right\}, \]
i.e. the set of properly embedded separating surfaces of genus $\mk g_0$ with $k$ boundary components in $\widetilde M$. We endow $\ms E$ with the smooth topology.

As done in \cite{Wang-Zhou23}*{Section 8.1} (see also \cite{Li-CPAM}), to produce free boundary min-max solutions, we only count area and volume for a pair $(\Sigma, \Omega)$ restricted to $M$. To be precise, given $(V,\Omega)\in \VC(\widetilde M)$ \cite{Wang-Zhou23}*{Definition 1.3}, we define 
\begin{equation}\label{eq:localized A^h}
\mc A^h_M(V,\Omega)= \|V\|(M)-\int_{\Omega\cap M} h\,\mr d \Vol.     
\end{equation}

Denote by $\mk{Is}$ the space of all boundary preserving isotopies on $\widetilde M$. Define 
\[\mk{Is}^{\mr{out}}:=\big\{\{\varphi_s\}\in \mk{Is}: M\subset \varphi_s(M) \text{ for all } s\in[0,1] \big\}\]
to be the isotopies in $\widetilde M$ that can push points out of the compact set $M$, but not into $M$; see also \cite{Li-CPAM}. Given an open subset $U\subset \widetilde M$, we define $\mk{Is}^{\mr{out}}(U)$ to be those in $\mk{Is}^{\mr{out}}$ that are supported in $U$.

Let $X\subset I(m, k_0)$ be a cubical complex and let $Z\subset X$ be a subcomplex. A continuous map $\psi: X\to C^\infty(\widetilde M, \widetilde M)$ is said to be \textit{an outward isotopic deformation}, if for each $x\in X$, there exists an outward isotopy $\{\varphi_{x, s}\}_{s\in[0,1]}\in \mk{Is}^{\mr {out}}$ such that $\psi_x := \psi(x)$ is equal to $\varphi_{x, 1}$.  Let $\Phi_0:X\to \ms E$ be a continuous map. A family $\bm \Xi$ of $(X,Z)$-sweepouts homotopic to $\Phi_0$ relative to $\Phi_0|_Z$ is said to be \textit{saturated}, if for any $\Phi\in \bm \Xi$, and any outward isotopic deformation $\psi: X\to C^\infty(\widetilde M, \widetilde M)$ with $\psi|_Z = \Id$,
\[ \Phi'(x) := (\psi_x)_\#\Phi(x) \quad \text{also belongs to $\bm \Xi$.} \]

For any saturated family $\bm\Xi$ of $(X,Z)$-sweepouts, we define
\begin{equation*}
\mf L^h_M(\bm\Xi)=\inf_{\Phi\in \bm\Xi}\max_{x\in X}\mc A^h_M \big(\Phi(x)\big).
\end{equation*}

Recall that a $C^{1,1}$ almost embedded surface $(\Sigma, \partial\Sigma)\subset (M, \partial M)$ is called a \textit{boundary} if there exists $\Omega\in \C(M)$, such that $\Sigma\res \interior(M) = \partial\Omega\res\interior(M)$\footnote{Since $h=0$ near $\partial M$, we only need this boundary structure in the interior $\interior(M)$.}, and a pair $(\Sigma, \Omega)$ is called a \textit{free $h$-boundary} if for any vector field $X\in \mk X(M, \Sigma)$\footnote{This means $X(q)\in T_q(\partial M)$ for all $q$ in a neighborhood of $\partial \Sigma$ in $\partial M$.} (see \cite{LZ16}*{page 501}), we have $\delta \A^h_M|_{\Sigma, \Omega}(X) = 0$. By the first variation formula, $\Sigma$ must meet $\partial M$ orthogonally along $\partial\Sigma$, and $\interior(\Sigma)$ is minimal near $\partial M$ because $h$ is assumed to vanish in a neighborhood of $\partial M$; see \cite{LZ16}*{Definition 2.8}. 

For the standard notions of ``{\it minimizing sequence},'' ``{\it strongly $\mc A^h$-stationary}'', ``{\it critical set},'' ``{\it min-max subsequence},'' ``{\it $(\mc A^h,\epsilon_j,\delta_j)$-almost minimizing},'' ``{\it  $L(m)$-admissible collection of annuli},'' we refer to \cite{Wang-Zhou23}. 

The notion of \textit{$(\A^h, \epsilon, \delta)$-outer almost minimizing in $U$} is specific to the free boundary setting, and is recorded in the following definition.

\begin{definition}\label{def:ep pmc almost minimizing}
Given $\epsilon, \delta >0$, an open set $U\subset \wti M$, and an embedded separating surface $(\Sigma, \Omega)\in \ms E$, we say that {\em $(\Sigma, \Omega)$ is $(\A^h, \epsilon, \delta)$-outer almost minimizing in $U$} if there does not exist any isotopy $\psi\in \mk{Is}^{out}(U)$, such that
\[ \A^h(\psi(t, \Sigma, \Omega)) \leq \A^h(\Sigma, \Omega) + \delta \quad \text{ for all } t\in[0,1];\]
\[  \A^h(\psi(1, \Sigma, \Omega)) \leq \A^h(\Sigma, \Omega) - \epsilon. \]
\end{definition}

We are ready to recall the min-max theorem for free boundary $h$-surfaces.

\begin{theorem}\label{thm:relative min-max in any compact domain} 
Let $\bm\Xi$ be a saturated family of $(X,Z)$-sweepouts relative to $\Phi_0|_Z$. Suppose 
\begin{equation*}
\mf L^h_M(\bm \Xi)>\max \left\{\max_{x\in Z}\mc A^h_M\big(\Phi_0(x)\big), 0\right\}.
\end{equation*}
Then there exist an integer $L=L(m)$ (depending only on the dimension of the large complex $I(m, k_0)$ into which $X$ is embedded), a minimizing sequence $\{\Phi_i\}_{i\in\mb N}\subset \bm \Xi$, and a strongly $\mc A^h$-stationary, $C^{1,1}$, free $h$-boundary $(\Sigma,\Omega)$ lying in the critical set $\mf C(\{\Phi_i\})$, that is, $(\Sigma,\Omega)$ is the $\ms F$-limit of some min-max subsequence $\{\Phi_{i_j}(x_j)\res M\}$, and 
\[ \mc A^h_M(\Sigma,\Omega)=\mf L^h_M(\bm\Xi). \]
Moreover, the min-max sequence $\{\Phi_{i_j}(x_j)\}_{j\in\mb N}$ can be chosen so that there exist $\epsilon_j,\delta_j\to 0$ such that $\Phi_{i_j}(x_j)$ is $(\mc A^h,\epsilon_j,\delta_j)$-outer almost minimizing in any $L(m)$-admissible collection of annuli $\mr{An}(p;s_1,r_1),\cdots, \mr{An}(p;s_L,r_L)\subset\wti M$. 
\end{theorem}

\begin{proof}
By adapting the tightening process from \cite{Wang-Zhou23}*{\S 2.2} and the combinatorial argument from \cite{Wang-Zhou23}*{\S 3.2} to the current free boundary setting, we obtain a stationary pair $(V, \Omega)\in \VC(M)$ which is free boundary stationary and satisfies the outer almost minimizing property.

The interior regularity was established by the last two authors in \cite{Wang-Zhou23}*{Theorem 8.1}, building on the $C^{1,1}$ regularity theory for minimizing $h$-surfaces developed by the first two authors in \cite{SS23} and by the last two authors in \cite{Wang-Zhou-C11-2023}. Note that $h=0$ in a neighborhood of $\partial M$. Hence the boundary regularity on the mean convex boundary $\partial \wti M$ follows from \cite{GrJo86}. The boundary regularity on $\partial_{other}M$ was established in \cite{Li-CPAM} and \cite{LZ16}. The last statement follows from the proof of \cite{Wang-Zhou23}*{Lemma 3.10}. 
\end{proof}

\begin{remark}\label{rem:stable} \leavevmode
\begin{itemize}
    \item By Theorem \ref{thm:relative min-max in any compact domain}, for any $L(m)$-admissible collection of annuli in $\widetilde M$, the surface $\Sigma$ is stable in at least one of them.
    \item When $\widetilde M = M$ (i.e. when all components of $\partial M$ are mean convex) and $h = 0$, the min--max theory in Theorem \ref{thm:relative min-max in any compact domain} reduces to that of \cite{GrJo86}.
    \item When $\widetilde M$ is closed (i.e. when no components of $\partial M$ are mean convex) and $h = 0$, the min--max theory in Theorem \ref{thm:relative min-max in any compact domain} coincides with that in \cite{Li-CPAM}.
\end{itemize}
\end{remark}

\begin{comment}
\subsection{Passing to limit and topological bound}
\begin{theorem}[\citelist{\cite{Li-CPAM}*{Theorem 9.1}\cite{Franz-Schulz-boundary-components}*{Theorem 1.8}}]\label{thm:topology bound}
With the same notations as in Theorem \ref{thm:classical multiplicity one with boundary}. Denote by $I_O$ (resp. $I_U$) the collection of $i$ such that $\Gamma_i$ is orientable (resp. non-orientable).
If $M$ orientable and $\partial M$ is strictly mean convex w.r.t. $g$,
then $\Gamma_i$ is orientable and
\begin{equation}\label{eq:genus bound1}
   \beta_1(\Gamma)\leq \beta_1(\Sigma_0),\quad   \sum_{i\in I_O}m_i\mk g(\Gamma_i)+\frac{1}{2}\sum_{i\in I_U}m_i(\mk g(\Gamma_i)-1)\leq \mk g(\Sigma_0),
\end{equation}
where $\mk g_0$ is the genus of $\Sigma_0$; $\beta_1(\Gamma')$ is the first Betti number of $\Gamma'$ for all surface-with-boundary $\Gamma'$. In particular, if $M$ is ball and $\Sigma$ is a disk, then each $\Gamma_i$ is either a disk or a sphere.
\end{theorem}

\end{comment}

\section{Multiplicity one theorem for free boundary Simon--Smith min-max theory}\label{sec:multi 1}
In this section, we prove the multiplicity one theorem for the free boundary Simon--Smith min-max theory, following the approach of \cite{Wang-Zhou23}*{\S6--7}. As a first step, we establish the convergence of free $\varepsilon h$-surfaces as $\varepsilon \to 0$. The key observation is that if a sequence of free $\varepsilon h$-surfaces converges to a free boundary minimal surface with multiplicity greater than one as $\varepsilon \to 0$, then the limiting free boundary minimal surface admits a nonnegative supersolution to the Jacobi equation. We choose a special function $h$ satisfying $h \equiv 0$ in a neighborhood of $\partial M$, so that the regularity theory in Theorem \ref{thm:relative min-max in any compact domain} applies.

\subsection{Existence of a supersolution} We extract a supersolution to the Jacobi equation under the assumption of convergence with multiplicity greater than one.
\begin{proposition} \label{prop:existence of supsolution}
Let $\Sigma \subset M$ be a properly embedded, two-sided, free boundary minimal surface. Let $h \in C^{\infty}(M)$ so that $h \equiv 0$ in a neighborhood of $\partial M$ and $h$ changes sign along $\Sigma \setminus \partial M$. Let $\{(\Sigma_k, \Omega_k)\}_{k\in\mb N}$ be a sequence of strongly $\mc A^{ \varepsilon_kh}$-stationary, $C^{1, 1}$, free $ \varepsilon_k h$-boundaries in $(M,g)$ with $\varepsilon_k\to 0$ as $k\to 0$. Suppose that $\Sigma_k$ converges as varifolds to $\Sigma$ with multiplicity $m\geq 2$. Suppose in addition that the convergence is $C^{1, 1}_{loc}$ away from a finite set $\mc Y$. Then $\Sigma$ admits a nonnegative function $\varphi\in W^{1,2}(\Sigma)$ with $\|\varphi\|_{L^2(\Sigma)}=1$ and a constant $c\geq 0$ satisfying 
\begin{equation}\label{eq:full ineq on Sigma} 	\int_{\Sigma}\langle\nabla\varphi,\nabla f\rangle - \big(\Ric(\nu,\nu) + |A^{\Sigma}|^2\big)\varphi f\,\mr d\mc H^2 -\int_{\partial \Sigma} A^{\partial M}(\nu,\nu)\varphi f\,\mr d\mc H^1  \geq \int_{\Sigma}2c hf\,\mr d\mc H^2, 
\end{equation}
for all  $f\in C^1(\Sigma) \text{ and } f\geq 0$.
Here $c=0$ if $m\geq 3$ is odd.
\end{proposition}
\begin{proof}
If $U$ is a relatively open subset of $\Sigma$ that is compactly contained in $\Sigma \setminus \mc Y$, then we can write (for $k$ sufficiently large) $\Sigma_k$ in a normal exponential neighborhood of $U$ as the normal exponential graphs of some functions $u_k^i \in C^{1,1}(U)$ so that
\[ u_k^1 \leq \hdots \leq u_k^m \]
and each $u_k^i$ converges to zero in $C^{1,1}(U)$ as $k \to \infty$. When $m \geq 2$, we set $\varphi_k := u_k^m - u_k^1 \geq 0$.

We decompose $\Sigma$ into a good boundary region and a good interior region. Let $U_B$ be a relatively open subset of $\Sigma$ so that $\Sigma \cap \partial M \subset U_B$ and $h\mid_{U_B} \equiv 0$. Let $U_I$ be a relatively open subset of $\Sigma$ compactly contained in $\Sigma \setminus \partial M$ so that $U_B \cup U_I = \Sigma$ and $h$ changes sign along $U_I$.

Following \cite{Wang-Zhou23}*{\S6}, there is a subsequence (not relabeled) and an exhaustion $\{U_I^k\}_{k \in \mb N}$ of $U_I \setminus \mc Y$ so that the sequence of renormalized functions $\widetilde{\varphi}_I^k := \varphi_k / \|\varphi_k\|_{L^2(U_I^k)}$ converges to a uniformly bounded $C^{1,\alpha}$ function $\varphi_I : U_I\setminus \mc Y \to [0, \infty)$ that extends to a nontrivial nonnegative function $\varphi_I \in W^{1,2}(U_I)$ which satisfies
\begin{equation}\label{eq:ineq in Int of Sigma} \int_{U_I}\langle\nabla\varphi_I,\nabla f\rangle - \big(\Ric(\nu,\nu) + |A^{\Sigma}|^2\big)\varphi_I f\,\mr d\mc H^2 \geq \int_{U_I}2c hf\,\mr d\mc H^2,  \end{equation}
for all $f \in C^1_c(U_I)$ with $f \geq 0$, where $c$ is a nonnegative constant (equal to 0 if $m \geq 3$ is odd).

By standard elliptic regularity, for any relatively open subset $U$ of $\Sigma$ that is compactly contained in $U_B \setminus \mc Y$, the graphs $u_k^i\mid_U$ are smooth and satisfy the minimal surface equation. Moreover, they converge to 0 in $C^{\infty}_\mathrm{loc}$.

\medskip
\noindent\emph{Case 1}: Suppose there is an exhaustion $\{U_B^k\}_{k \in \mb N}$ of $U_B \setminus \mc Y$ and a subsequence (not relabeled) so that $\varphi_k\mid_{U_B^k} \equiv 0$. Then $\{U^k:= U_I^k \cup U_B^k\}_{k\in \mb N}$ is an exhaustion of $\Sigma \setminus \mc Y$, and the renormalized functions $\widetilde{\varphi}_I^k := \varphi_k / \|\varphi_k\|_{L^2(U_I^k)}$ converge to a uniformly bounded $C^{1,\alpha}$ function $\varphi : \Sigma \setminus \mc Y \to [0, \infty)$ that extends to a nonnegative function $\varphi \in W^{1,2}(\Sigma)$ which satisfies \eqref{eq:full ineq on Sigma}. Indeed, this conclusion follows from the arguments in $U_I$ above (i.e.\ \eqref{eq:ineq in Int of Sigma}) and the fact that $\varphi \equiv 0$ on $U_B \setminus \mc Y$.

\medskip
\noindent\emph{Case 2}: Suppose there is an exhaustion $\{U_B^k\}_{k \in \mb N}$ of $U_B \setminus \mc Y$, and a subsequence (not relabeled) so that $\varphi_k\mid_{U_B^k} \not\equiv 0$. Since the sheets are minimal surfaces over $U_B^k$, the strong maximum principle implies that $\varphi_k\mid_{U_B^k} > 0$. Following \cite{Ambrozio-Carlotto-Sharp18}*{\S6}, there is a further subsequence (not relabeled) and a point $z \in U_B \setminus (\mc Y \cup \partial M)$ that is in $U_B^k$ for all $k$ so that the sequence of renormalized functions $\widetilde{\varphi}_B^k := \varphi_k / \varphi_k(z)$ converges to a uniformly bounded smooth function $\varphi_B : U_B \setminus \mc Y \to [0, \infty)$ which extends to a \emph{positive} (by the strong maximum principle) smooth function $\varphi_B \in C^{\infty}(U_B)$ satisfying
\begin{equation}\label{eq:ineq on boundary Sigma} 	\int_{U_B}\langle\nabla\varphi_B,\nabla f\rangle - \big(\Ric(\nu,\nu) + |A^{\Sigma}|^2\big)\varphi_B f\,\mr d\mc H^2 -\int_{\partial \Sigma} A^{\partial M}(\nu,\nu)\varphi_B f\,\mr d\mc H^1  = 0, 
\end{equation}
for all $f\in C^1_c(U_B)$\footnote{Here we mean that $f$ is supported in a relatively compact subset of $\Sigma$ contained in $U_B$ (meaning that $f$ can be nonzero on $\partial M \cap U_B$).} and $f\geq 0$.

\smallskip
\emph{Subcase A}: Suppose that there is a subsequence (not relabeled) so that
\[ \lim_{k \to \infty} \frac{\|\varphi_k\|_{L^2(U_I^k)}}{\varphi_k(z)} = 0. \]
Along the exhaustion $\{U^k := U_I^k \cup U_B^k\}_{k \in \mb N}$, the sequence of renormalized functions $\widetilde{\varphi}_B^k := \varphi_k/\varphi_k(z)$ converges to a uniformly bounded smooth function $\varphi_B : \Sigma \setminus \mc Y \to [0, \infty)$ that extends to a nontrivial smooth function $\varphi_B \in C^{\infty}(\Sigma)$ that satisfies
\begin{equation}\label{eq:subcase A} 	\int_{\Sigma}\langle\nabla\varphi_B,\nabla f\rangle - \big(\Ric(\nu,\nu) + |A^{\Sigma}|^2\big)\varphi_B f\,\mr d\mc H^2 -\int_{\partial \Sigma} A^{\partial M}(\nu,\nu)\varphi_B f\,\mr d\mc H^1  = 0
\end{equation}
for all $f \in C^1(\Sigma)$ with $f \geq 0$. Indeed, $\varphi_B\mid_{U_I} \equiv 0$ since this rescaling is faster than the rescaling chosen for the interior piece, and \eqref{eq:subcase A} then follows from \eqref{eq:ineq on boundary Sigma}. Renormalizing $\varphi_B$ to have $L^2$ norm 1 concludes the proposition in this case.

\smallskip
\emph{Subcase B}: Suppose that there is a subsequence (not relabeled) so that
\[ \lim_{k \to \infty} \frac{\|\varphi_k\|_{L^2(U_I^k)}}{\varphi_k(z)} = \infty. \]
Along the exhaustion $\{U^k := U_I^k \cup U_B^k\}_{k \in \mb N}$, the sequence of renormalized functions $\widetilde{\varphi}_I^k := \varphi_k/\|\varphi_k\|_{L^2(U_I^k)}$ converges to a uniformly bounded $C^{1,\alpha}$ function $\varphi_I : \Sigma \setminus \mc Y \to [0, \infty)$ that extends to a nontrivial nonnegative function $\varphi_I \in W^{1,2}(\Sigma)$ such that
\begin{equation}\label{eq:subcase B} 	\int_{\Sigma}\langle\nabla\varphi_I,\nabla f\rangle - \big(\Ric(\nu,\nu) + |A^{\Sigma}|^2\big)\varphi_I f\,\mr d\mc H^2 -\int_{\partial \Sigma} A^{\partial M}(\nu,\nu)\varphi_I f\,\mr d\mc H^1  \geq \int_{\Sigma}2c hf\,\mr d\mc H^2
\end{equation}
for all $f \in C^1(\Sigma)$ with $f \geq 0$. Indeed, $\varphi_I\mid_{U_B} \equiv 0$ since this rescaling is faster than the rescaling chosen for the boundary piece, and \eqref{eq:subcase B} then follows from \eqref{eq:ineq in Int of Sigma}. Renormalizing $\varphi_I$ to have $L^2$ norm 1 concludes the proposition in this case.

\smallskip
\emph{Subcase C}: Suppose there is a subsequence (not relabeled) so that 
\[ \lim_{k \to \infty} \frac{\|\varphi_k\|_{L^2(U_I^k)}}{\varphi_k(z)} = a \in (0, \infty). \]
Along the exhaustion $\{U^k := U_I^k \cup U_B^k\}_{k \in \mb N}$, the sequence of renormalized functions $\widetilde{\varphi}_I^k := \varphi_k/\|\varphi_k\|_{L^2(U_I^k)}$ converges to a uniformly bounded $C^{1,\alpha}$ function $\varphi: \Sigma \setminus \mc Y \to [0, \infty)$ that extends to a nontrivial nonnegative function $\varphi\in W^{1,2}(\Sigma)$ such that
\begin{equation}\label{eq:subcase C} 	\int_{\Sigma}\langle\nabla\varphi,\nabla f\rangle - \big(\Ric(\nu,\nu) + |A^{\Sigma}|^2\big)\varphi f\,\mr d\mc H^2 -\int_{\partial \Sigma} A^{\partial M}(\nu,\nu)\varphi f\,\mr d\mc H^1  \geq \int_{\Sigma}2c hf\,\mr d\mc H^2
\end{equation}
for all $f \in C^1(\Sigma)$ with $f \geq 0$. Indeed, by the choice of rescaling we have $\varphi\mid_{U_I} \equiv \varphi_I\mid_{U_I}$ and $\varphi\mid_{U_B} \equiv a^{-1}\varphi_B\mid_{U_B}$, so we combine \eqref{eq:ineq in Int of Sigma} and \eqref{eq:ineq on boundary Sigma} to get \eqref{eq:subcase C}. Renormalizing $\varphi$ to have $L^2$ norm 1 concludes the proposition in this case.
\end{proof}

\subsection{Multiplicity one for relative Simon--Smith min-max theory in compact manifolds with boundary}
We establish a multiplicity one result for the relative Simon--Smith min-max theory.

\begin{definition}\label{def:property R'}
    An embedded free boundary minimal surface $\Sigma$ satisfies Property $\mathbf{(R')}$ with constant $L$ if for every $L$-admissible collection $\mathscr{C}$ of annuli, $\Sigma$ is stable (for area) in at least one annulus in $\mathscr{C}$.
\end{definition}
Recall that the free boundary minimal surfaces produced by the min-max theory satisfy Property $\mathbf{(R')}$; see Theorem \ref{thm:relative min-max in any compact domain}. This property plays a crucial role in the compactness theory.
\begin{theorem}\label{thm:compactness and jacobi fields}
    Let $L$ be a positive integer and $\Lambda > 0$.
    Let $\{g_i\}$ be a sequence of smooth Riemannian metrics on $M$ converging smoothly to $g$.
    Let $\Sigma_k$ be a sequence of embedded free boundary minimal surfaces in $(M,\partial M,g_i)$ satisfying
    \begin{itemize}
        \item $\mathcal{H}^2(\Sigma_k;g_i) \leq \Lambda$,
        \item Property $\mathbf{(R')}$ with constant $L$.
    \end{itemize}
    Then a subsequence of $\{\Sigma_k\}$ converges to an almost properly embedded free boundary minimal surface $\Sigma$ in the sense of varifolds (possibly with integer multiplicity). Furthermore, 
    \begin{itemize}
    \item if $\Sigma_k \neq \Sigma$ for infinitely many $k$ along the convergent subsequence, then $\Sigma$ is degenerate;
    \item there is a finite set $\mc Y\subset \Sigma$ such that $\Sigma_k$ converges to $\Sigma$ locally smoothly away from $\mc Y$;
    \item if the convergence has multiplicity greater than $1$, then $\Sigma$ is stable.
    \end{itemize}
\end{theorem}
\begin{proof}
The compactness is essentially given by \cite{Sharp17} for interiors and \cite{Ambrozio-Carlotto-Sharp18} for free boundaries. The degeneration of the limit free boundary minimal surfaces is proved in \cite{GWZ18} and \cite{Wang19} for general case; see also \cite{Ambrozio-Carlotto-Sharp18} for ambient manifolds with mean convex case.
\end{proof}

\begin{theorem}[Multiplicity one for relative min-max]\label{thm:rel_mult_one}
    Let $X \subset I(m, k_0)$ be a cubical complex and $Z \subset X$ a subcomplex. Let $\Phi_0 : X \to \mathscr{E}$ be a continuous map and let $\Pi$ be the $(X, Z)$-homotopy class of $\Phi_0$. Assume
    \[ \mathbf{L}_M(\Pi) > \max_{x \in Z} \mathcal{H}^2(\Phi_0(x) \cap M). \]
    Then there exists a pairwise disjoint collection of connected, compact, smooth, almost properly embedded minimal surfaces with (possibly empty) free boundary $\Gamma = \cup_{i=1}^N \Gamma_i$ in $M$ and positive integers $\{m_i\}_{i=1}^N$, so that
    \[ \mathbf{L}_M(\Pi) = \sum_{i=1}^N m_i\mathcal{H}^2(\Gamma_i), \]
    and
    \begin{enumerate}
        \item if $\Gamma_i$ is two-sided and unstable, then $m_i = 1$,
        \item if $\Gamma_i$ is one-sided, then the connected double cover of $\Gamma_i$ is stable.
    \end{enumerate}
    Furthermore, if $M$ is orientable and has mean convex boundary, then
\begin{equation}
    \beta_1(\Gamma)\leq \beta_1(\Sigma_0),\quad   \sum_{i\in I_O}m_i\mk g(\Gamma_i)+\frac{1}{2}\sum_{i\in I_U}m_i(\mk g(\Gamma_i)-1)\leq \mk g(\Sigma_0),
\end{equation}
where $\beta_1(\Gamma)$ is the first Betti number of $\Gamma$, $\mk g(\Gamma)$ is the genus of $\Gamma$, and $I_O$ (resp. $I_U$) is the collection of $i$ such that $\Gamma_i$ is orientable (resp. non-orientable).
\end{theorem}
\begin{proof}
    We replicate the proof of \cite{Wang-Zhou23}*{Theorem 7.2}. Suppose that $(M, g)$ is strongly bumpy, meaning that every embedded free boundary minimal surface is non-degenerate and proper. By \cite{SWZ}*{Theorem 1.3}, the set of such metrics is $C^{\infty}$ generic in the Baire sense. For $\Lambda>0$ (we ultimately take $\Lambda = \mathbf{L}_M(\Pi) + 1$), let $\mathcal{M}(\Lambda)$ denote the set of embedded free boundary minimal surfaces $\Gamma \subset M$ satisfying $\mathcal{H}^2(\Gamma) \leq \Lambda$ and Property $\mathbf{(R')}$ with constant $L(m)$. Since $(M, g)$ is strongly bumpy, Theorem \ref{thm:compactness and jacobi fields} implies that $\mathcal{M}(\Lambda)$ is a finite set, which we can enumerate as $\mathcal{M}(\Lambda) = \{S_1, \hdots, S_{\alpha}\}$. Select points $p_1, \hdots, p_{\alpha}$ and $q_1, \hdots, q_{\alpha}$ in $\interior (M)$ so that $p_i, q_i \in S_j$ if and only if $i = j$. Choose $r>0$ sufficiently small so that
    \begin{itemize}
        \item $B_r(p_i)$ and $B_r(q_i)$ are contained in the interior of $M$,
        \item $B_r(p_1), \hdots, B_r(p_{\alpha}), B_r(q_1),\hdots, B_r(q_{\alpha})$ are pairwise disjoint,
        \item $B_r(p_i) \cup B_r(q_i)$ intersects $S_j$ if and only if $i = j$,
        \item $B_r(p_i) \cap S_i$ and $B_r(q_i) \cap S_i$ are embedded disks for all $i$.
    \end{itemize}
    We construct the prescribing function $h$ exactly as in \cite{Wang-Zhou23}*{Theorem 7.2}. Since $h$ will vanish outside the balls $\{B_r(p_i)\}$ and $\{B_r(q_i)\}$ which are disjoint from $\partial M$, $h$ vanishes in a neighborhood of $\partial M$. Hence, Proposition \ref{prop:existence of supsolution} applies in place of \cite{Wang-Zhou23}*{Proposition 6.4}. The rest of the proof now follows exactly as in \cite{Wang-Zhou23}*{Theorem 7.2}. For the sake of completeness, we provide a detailed proof here.

We take a smooth function $h:M\to \mb [-1,1]$ satisfying that for all $i=1,\dots ,\alpha$,
\begin{enumerate}
\item $h=0$ outside $\cup_i(B_r(p_i)\cup B_r(q_i))$;
\item $h> 0$ in $B_{r/2}(p_i)\cap  S_i$ and $h<0$ in $B_{r/2}(q_i)\cap  S_i$;
\item if $ S_i$ is two-sided, then $\int_{ S_i}h\phi_i\,\mr d\mc H^2=0$, where $\phi_i$ is the first eigenfunction of the Jacobi operator on $ S_i$;
\item if $ S_i$ is one-sided, then $\int_{\wti  S_i}h\phi_i\,\mr d\mc H^2=0$, where $\phi_i$ is the first eigenfunction of the Jacobi operator on $\wti  S_i$ and $\wti S_i$ is the connected double cover of $ S_i$.
	\end{enumerate} 
 
Now we choose $\varepsilon_k\to 0$. Then for sufficiently large $k$, we have 
\[     
\mf L^{\varepsilon_k h}(\Pi)>\max\left\{\max_{x\in Z}\mc A^{\varepsilon_kh}\big(\Phi_0(x)\big),0\right\}.
\] 
Applying Theorem \ref{thm:relative min-max in any compact domain} to the $\A^{\varepsilon_k h}$-functional for each $k$, we obtain a min-max pair $(V_k, \Omega_k)\in \VC(M)$, which is a strongly $\A^{\varepsilon_kh}$-stationary, $C^{1,1}$ $(\varepsilon_k h)$-boundary with $\A^{ \varepsilon_kh}(V_k, \Omega_k) = \mf L^{\varepsilon_k h}(\Pi)$. Moreover, by Theorem \ref{thm:relative min-max in any compact domain} (see Remark \ref{rem:stable}), for any $L(m)$-admissible collection of annuli in $\widetilde{M}$, the min-max pair is stable for $\A^{\varepsilon_kh}$ in at least one annulus. Then by Theorem \ref{thm:compactness and jacobi fields} and \cite{Wang-Zhou23}*{Proposition 5.1 and Theorem 5.2}, up to subsequence, $V_k$ (not relabelled) converges as varifolds to $V_\infty$ with $V_\infty=\sum_{i=1}^N m_i[\Gamma_i]$, where $\{\Gamma_i\}$ is a pairwise disjoint collection of connected, compact, embedded, free boundary minimal surfaces. Furthermore, there is a finite set $\mathcal{Y} \subset M$ so that the convergence is $C^{1,1}_{\mathrm{loc}}$ in $M \setminus \mathcal{Y}$. Since $ \varepsilon_k\to0$, we have
\[ \lim_{k\to\infty}\mf L^{ \varepsilon_kh}(\Pi) = \mf L(\Pi), 
\]
which yields that $\mc H^2 (\Gamma_i)\leq \Lambda$. Furthermore, we know that $\Gamma_\infty = \cup \Gamma_i$ satisfies Property {\bf(R')} (Definition \ref{def:property R'}). Hence $\Gamma_i$ is one of $S_1,\cdots, S_\alpha$. By relabelling, we assume that $\Gamma_i=S_i$ for $i=1,\dots,N$.  Observe that by the construction of $h$, the sign of $h$ changes on $\Gamma_i$. Since the convergence is locally $C^{1,1}$ away from finitely many points, and $(V_k,\Omega_k)$ is strongly $\mc A^{ \varepsilon_k h}$-stationary, then by Proposition \ref{prop:existence of supsolution}, any two-sided connected component $\Gamma_i$ with multiplicity $m_i\geq 2$ admits a nontrivial and nonnegative $\varphi\in W^{1,2}(\Gamma_i)$ such that for all non-negative $f\in C^1(\Gamma_i)$
\[
\int_{\Gamma_i}\langle\nabla\varphi,\nabla f\rangle- \big(\Ric(\nu,\nu) +|A^{\Gamma_i}|^2\big)\varphi f\,\mr d\mc H^n-\int_{\partial \Gamma_i} A^{\partial M}(\nu,\nu)\varphi f\,\mr d\mc H^1 \geq  \int_{\Gamma_i}2c hf\,\mr d\mc H^n,\]
for some constant $c\geq 0$. 
Let $\phi_i$ be the first eigenfunction of the Jacobi operator of $\Gamma_i$. Then $\phi_i>0$ and 
\[
 L_{\Gamma_i}\phi_i=\lambda_1(\Gamma_i) \phi_i \text{ in } \Gamma_i; \quad \frac{\partial \phi_i}{\partial \eta}= A^{\partial M}(\nu,\nu) \phi_i \text{ on } \partial \Gamma_i.
\]
Plugging this to the above inequality, we obtain 
\begin{align*}
	0=\int_{\Gamma_i} 2ch\phi_i\,\mr d\mc H^2 \leq \int_{ \Gamma_i}\varphi L_{\Gamma_i}\phi_i\,\mr d\mc H^2 = \lambda_1(\Gamma_i) \int_{\Gamma_i}\varphi\phi_i\,\mr d\mc H^2.
	\end{align*}
Recall that $\phi_i>0$ everywhere and $\varphi\geq 0$ with $\varphi>0$ somewhere. It follows that $\int_{\Gamma_i} \varphi \phi_i\,\mr d\mc H^2>0$. Thus we conclude that the first eigenvalue $\lambda_1(\Gamma_i)$ is non-negative, that is, if $\Gamma_i$ is two-sided and $m_i\geq 2$, then $\Gamma_i$ is stable. This proves the first item.

For one-sided connected component $\Gamma'\subset \spt\|V_\infty\|$, the same argument gives that the double cover of $\Gamma'$ is stable.

Note that by the choice of $h$, for each $\Gamma_i$, we have that  $h=0$ outside two disjoint balls $B_r(p_i)$ and  $B_r(q_i)$. Moreover, $B_r(p_i)\cap \Gamma_i$ and $B_r(q_i)\cap \Gamma_i$ are both disks. Hence the desired genus bound follows from   
\citelist{\cite{Li-CPAM}*{Theorem 9.1}\cite{Franz-Schulz-boundary-components}*{Theorem 1.8}}.

\medskip

For the general case when $g$ is not strongly bumpy, by \cite{SWZ}*{Theorem 1.3}, one can take a sequence of bumpy metrics $g_i$ converging to $g$ in $C^3$. Then the theorem follows from the conclusion for bumpy metrics, Property {\bf(R')}, and the standard compactness theory for embedded free boundary minimal surfaces; see Theorem \ref{thm:compactness and jacobi fields}. The topological bounds follow by applying \citelist{\cite{Li-CPAM}*{Theorem 9.1}\cite{Franz-Schulz-boundary-components}*{Theorem 1.8}} to a diagonal min-max sequence.
\end{proof}

\subsection{Multiplicity one for classical Simon--Smith min-max theory in compact manifolds with boundary}

We use the double cover lifting argument from \cite{Wang-Zhou23}*{\S7.2} to adapt the multiplicity one result of the previous subsection to the classical Simon--Smith min-max theory.

Let $\Sigma_0$ be a compact surface of genus $\mk g_0$ and with $k$ connected boundary components. Define
\[ \mathscr{X}(\Sigma_0) = \{\phi(\Sigma_0)\mid \phi : \Sigma_0 \to \widetilde{M} \text{\ is a proper separating embedding of\ } \Sigma_0 \text{\ in\ } \widetilde{M} \}, \]
and
\[ \mathscr{Y}(\Sigma_0) = \{\phi(\Sigma_0)\mid\phi : \Sigma_0 \to \tilde{M} \text{\ is a smooth map whose image is a 0-or 1-dimensional graph}\}. \]
We write
\[ \overline{\mathscr{X}}(\Sigma_0) = \mathscr{X}(\Sigma_0) \cup \mathscr{Y}(\Sigma_0). \]
Endow $\overline{\mathscr{X}}(\Sigma_0)$ with the \emph{unoriented} smooth topology for immersions.

Let $\widetilde{\overline{\mathscr{X}}}(\Sigma_0)$ (respectively $\widetilde{\mathscr{X}}(\Sigma_0)$ and $\widetilde{\mathscr{Y}}(\Sigma_0)$) denote the same sets endowed with the \emph{oriented} smooth topology. Note that $\widetilde{\overline{\mathscr{X}}}(\Sigma_0)$ is a nontrivial double cover of $\overline{\mathscr{X}}(\Sigma_0)$. Let $\pi : \widetilde{\overline{\mathscr{X}}}(\Sigma_0) \to \overline{\mathscr{X}}(\Sigma_0)$ be the covering map (given by forgetting the orientation). Let $\overline{\lambda} \in H^1(\overline{\mathscr{X}}(\Sigma_0), \Z_2)$ be the dual to the nontrivial element in $\pi_1(\overline{\mathscr{X}}(\Sigma_0))$ associated to the covering map $\pi$.

Let $X \subset I(m,k_0)$ be a cubical complex and $Z \subset X$ a subcomplex. Fix a continuous map
\[ \Phi_0: X \to \overline{\mathscr{X}}(\Sigma_0) \text{\ \ satisfying\ \ } \Phi_0(Z) \subset \mathscr{Y}(\Sigma_0). \]
We call such a map a \emph{sweepout}. Let $\Pi$ denote the set of all continuous maps homotopic to $\Phi_0$ relative to $\Phi_0\mid_{Z}$, which we call the \emph{$(X,Z)$-homotopy class of $\Phi_0$}. We define
\[ \mathbf{L}_M(\Pi) = \inf_{\Phi \in \Pi} \sup_{x \in X} \mathcal{H}^2(\Phi(x) \cap M). \]

\begin{theorem}[Multiplicity one for classical min-max]\label{thm:classical M1}
    Let $X \subset I(m, k_0)$ be a cubical complex and $Z \subset X$ a subcomplex. Let $\Phi_0 : X \to \overline{\mathscr{X}}(\Sigma_0)$ be a sweepout and let $\Pi$ be the $(X, Z)$-homotopy class of $\Phi_0$. Assume
    \[ \mathbf{L}_M(\Pi) > \max_{x \in Z} \mathcal{H}^2(\Phi_0(x) \cap M) = 0. \]
    Then there exists a pairwise disjoint collection of connected, compact, smooth, almost properly embedded minimal surfaces with (possibly empty) free boundary $\Gamma = \cup_{i=1}^N \Gamma_i$ in $M$ and positive integers $\{m_i\}_{i=1}^N$, so that
    \[ \mathbf{L}_M(\Pi) = \sum_{i=1}^N m_i\mathcal{H}^2(\Gamma_i), \]
    and
    \begin{enumerate}
        \item if $\Gamma_i$ is two-sided and unstable, then $m_i = 1$,
        \item if $\Gamma_i$ is one-sided, then the connected double cover of $\Gamma_i$ is stable.
    \end{enumerate}
Furthermore, if $M$ is orientable and has mean convex boundary, then
\begin{equation}
    \beta_1(\Gamma)\leq \beta_1(\Sigma_0),\quad   \sum_{i\in I_O}m_i\mk g(\Gamma_i)+\frac{1}{2}\sum_{i\in I_U}m_i(\mk g(\Gamma_i)-1)\leq \mk g(\Sigma_0),
\end{equation}
where $\beta_1(\Gamma)$ is the first Betti number of $\Gamma$, $\mk g(\Gamma)$ is the genus of $\Gamma$, and $I_O$ (resp. $I_U$) is the collection of $i$ such that $\Gamma_i$ is orientable (resp., non-orientable).
\end{theorem}
\begin{proof}
The proof in the Almgren--Pitts framework is originally from \cite{Zhou19}*{Theorem 5.2}, which was generalized in \cite{SWZ}*{Theorem 4.7} to the free boundary setting. For the case of Simon--Smith min-max theory, the theorem follows the proof of \cite{Wang-Zhou23}*{Theorem 7.3} essentially verbatim, replacing \cite{Wang-Zhou23}*{Theorem 7.2} with our Theorem \ref{thm:rel_mult_one}.
\end{proof}

\section{Existence of free boundary minimal disks} \label{sec:disks}
\subsection{Sweepouts in three-balls}\label{subsec:three sweepouts}
In this section, we will first recall the fact that the three-ball always admits a nontrivial $k$-parameter ($k=1,2,3$) sweepout of disks. 
These sweepouts will be used to prove the existence of free boundary minimal disks by applying the Multiplicity One Theorem for Simon--Smith min-max theory (Theorem \ref{thm:classical M1}). To show that these sweepouts will produce three different surfaces, we will recall the Lusternik–Schnirelmann theory.

Let $\mb{B}^3\subset\mb R^3$ be the standard unit round ball, and $x_1, x_2, x_3$ be the three coordinate functions. Consider the spaces
\[
\ms X:=\{\phi(\mb D)\big|\phi: \mb D\to \mb B^3 \text{ is a smooth proper embedding}\},
\]
and 
\[
\ms Y:=\{\phi(\mb D)\big|\phi:\mb D\to\mb B^3 \text{ is a smooth map whose image is a point or an interval}\}.
\]
Denote by $\oli{\ms X}=\ms X\cup \ms Y$ endowed with un-oriented smooth topology. 

For each $i=1, 2, 3$, let $\ms P_i$ be the collection of continuous maps $\Phi: X\to \oli{\ms X}$,  
with $\Phi(Z)\subset \ms Y$, such that there exists $\oli\lambda\in H^1(\oli {\ms X},\ms Y;\mb Z_2)$ for which $\Phi^*(\oli\lambda)^i \neq 0 \in H^i(X, Z_0;\mb Z_2)$.

%By Hatcher’s proof of Smale’s conjecture \cite{Hat83}*{Appendix (14)}, $\oli{\ms X}$ is homotopy equivalent to $\mb {RP}^2$; see \cite{HK19}*{Section 2}.Denote by $\ms P_i$, $i=1, 2, 3$, the collection of continuous $\Phi: (X, Z_0)\to \oli{\ms X}$, so that $\Phi(Z_0)\subset \ms Y$, and\[  \Phi^*(H^i(\oli{\ms X},\ms Y;\mb Z_2))\neq 0\in  H^i(X, Z_0;\mb Z_2). \] 

\begin{comment}
Next we describe three explicit sweepouts that detect the four nontrivial cohomology classes in $H^*(\oli{\ms X}, \partial\ms X, \mb Z_2)$. 
We use $[-1, 1]\ttimes \mb{RP}^2$ to denote the twisted $[-1,1]$-bundle over $\mb{RP}^2$, and $[a_0, a_1,a_2,a_3]$ to denote a point in $[-1, 1]\ttimes\mb{RP}^2$, that is $a_0\in [-1, 1]$, $a_1^2+a_2^2+ a_3^2=1$, and $(a_0, a_1, a_2, a_3)$ is identified with $(-a_0, -a_1, -a_2, -a_3)$. 
When ${\color{red}a_0\neq \pm 1}$, we denote
\[ \mc G([a_0, a_1,a_2,a_3]):=\{a_1x_1+a_2x_2+a_3x_3= a_0\}\cap \mb B^3;\]
when $a_0=\pm 1$, $\mc G(a_0, a_1, a_2, a_3)$ denotes a point given by $\pm(a_1, a_2, a_3)\in \partial\mb B^3$.
\end{comment}

Next we describe three explicit sweepouts that belong to $\ms P_i$ for each $i=1, 2, 3$.  
We use $[-1, 1]\ttimes \mb{RP}^2$ to denote the twisted $[-1,1]$-bundle over $\mb{RP}^2$, and $[a_0, a_1,a_2,a_3]$ to denote a point in $[-1, 1]\ttimes\mb{RP}^2$; that is $a_0\in [-1, 1]$, $a_1^2+ a_2^2 + a_3^2=1$, and $(a_0, a_1, a_2, a_3)$ is identified with $(-a_0, -a_1, -a_2, -a_3)$. 
When $a_0\neq \pm 1$, we denote
\[ \mc G([a_0, a_1,a_2,a_3]):=\{a_1x_1+a_2x_2+a_3x_3= a_0\}\cap \mb B^3;\]
when $a_0=\pm 1$, $\mc G(a_0, a_1, a_2, a_3)$ denotes a point given by $\pm(a_1, a_2, a_3)\in \partial\mb B^3$.

To detect the nontrivial cohomology classes, we construct explicit sweepouts parametrized by the twisted $[-1,1]$-bundle over $\mb{RP}^{i-1}$. We now define three maps:
\begin{gather*} 
\Psi_1: [-1,1]\ttimes \mb{RP}^0\to \oli{\ms X}, \quad a_0 \longmapsto \mc G(a_0,1,0,0);\\
\Psi_2: [-1,1]\ttimes \mb{RP}^1\to \oli{\ms X}, \quad [a_0,a_1,a_2]\longmapsto \mc G(a_0,a_1,a_2,0);\\
\Psi_3: [-1,1]\ttimes\mb{RP}^2\to \oli{\ms X}, \quad [a_0,a_1,a_2,a_3]\longmapsto \mc G(a_0,a_1,a_2,a_3).
\end{gather*}
For simplicity, we use $\mc X_i$ to denote $[-1,1]\ttimes\mb {RP}^{i-1}$, (and as compared with our definition of $\ms P_i$, $Z_0=\partial \mc X_i$ for each $i$.) Let $f: [-1, 1] \to \mc X_i$ be an arbitrary embedding of a fiber of the $[-1,1]$-bundle.

We now show that $\Psi_i\in\ms P_i$ for $i=1,2,3$. Denote by $\iota: \oli{\ms X} \to \mc Z_2( B^3,\partial B^3; \mb Z_2)$ the natural inclusion map into the space of mod-2 relative integral cycles. Note that $\iota(\mc Y) = \{0\}$, that is, the image of each element in $\mc Y$ is a zero relative cycle.  Consider the chain of maps:
\[
\begin{tikzcd}
({[-1, 1]}, \partial{[-1, 1]}) \arrow[r, "f"]
&(\mc X_i, \partial \mc X_i) \arrow[r, "\Psi_i"]
&(\oli{\ms X}, \ms Y) \arrow[r, "\iota"] 
&(\mc Z_2(B^3,\partial B^3;\mb Z_2),\{0\});
\end{tikzcd}
\]
the composition map $\iota\circ \Psi_i\circ f: ([-1, 1], \partial{[-1, 1]}) \to (\mc Z_2(B^3,\partial B^3;\mb Z_2),\{0\})$ is then a sweepout in the sense of Almgren; see \cite{Alm62}, \cite{Marques-Neves16}*{Definition 3.4} and \cite{Zhou15}*{Theorem 5.8}. Therefore, 
we know that $(\iota\circ \Psi_i\circ f)^*: H^1(\mc Z_2(B^3,\partial B^3;\mb Z_2),\{0\}; \mb Z_2) \to H^1([-1, 1], \partial{[-1, 1]}; \mb Z_2) = \mb Z_2$ is nontrivial, and by the chain of pull-back maps:
\[  
\begin{tikzcd}
H^1(\mc Z_2(B^3,\partial B^3;\mb Z_2),\{0\};\mb Z_2) \arrow[r,"\iota^*"] &H^1(\oli{\ms X},\ms Y;\mb Z_2)\arrow[r,"\Psi_i^*"] &H^1(\mc X_i,\partial \mc X_i;\mb Z_2)\\ \arrow[r,"f^*"]
&H^1([-1, 1], \partial[-1, 1]; \mb Z_2),
\end{tikzcd}
\]
we also know that $\Psi_i^*: H^1(\oli{\ms X}, \ms Y; \mb Z_2) \to H^1(\mc X_i, \partial\mc X_i; \mb Z_2)$ is nontrivial. This together with the structure of the relative cohomology ring 
\[
H^*(\mc X_i,\partial \mc X_i;\mb Z_2)\simeq \mb Z_2[\alpha]/[\alpha^{i+1}] \] 
implies that $\Psi_i\in \ms P_i$ for $i=1,2,3$.

Recall that 
 \[ \mf L(\ms P_i):=\inf_{\Phi\in \ms P_i}\sup_{x\in \mr{dom} \Phi}\mc H^2(\Phi(x)).\] 

The next result follows from Lusternik–Schnirelmann theory.
\begin{lemma}[\citelist{\cite{HK19}*{Theorem 5.2} \cite{Wang-Zhou23}*{Lemma 8.3}}]\label{lem:LS theory}
Let $(M,g)$ be a Riemannian 3-ball with strictly mean convex boundary.
Suppose that $(B^3,\partial B^3,g)$ contains only finitely many embedded minimal two-spheres and free boundary minimal disks. Then 
\[   0<\mf L(\ms P_1)<\mf L(\ms P_2)<\mf L(\ms P_3).  \]
\end{lemma}

\subsection{Non-compact manifolds with cylindrical ends}\label{subsec:cylindrical ends}
In this section, we recall the construction of non-compact manifolds with cylindrical ends by A. Song in \cite{Song18}*{Section 2.2}.
Let $(N,\partial_{rel}N\cup T,g)$ be a compact Riemannian manifold with disjoint boundaries $T$ and $\partial_{rel}N$.
Suppose that $T$ is a closed, embedded, stable minimal surface with \textit{a contracting neighborhood}, that is, there exist $\mu>0$ and a map 
\[ \varphi:T\times [0,\mu ]\to N,\]
so that $\varphi$ is a diffeomorphism to its image, $\varphi(T\times \{0\})=\partial N$, and for all $t\in (0,\hat t]$, $\varphi(T\times \{t\})$ has non-zero mean curvature vector\footnote{Here the mean curvature vector is defined as $-\dv(\nu)\nu$ for a choice of unit normal $\nu$.} pointing towards $T$. We endow $T\times [0,+\infty)$ with the product metric. Let $\ms C(N)$ be the non-compact manifold
\[   N\cup (T\times [0,+\infty))  \]
by identifying $T$ with $T\times \{0\}$. We endow it with the metric $\hat g$ such that $\hat g=g$ on $N$ and is equal to the product metric on $T\times [0,\infty)$. Note that $\hat g$ is Lipschitz continuous. 

Now we approximate $\ms C(N)$ by compact manifolds as follows. %Fix $q\in N\setminus \partial N$. 
Let $N_\epsilon:=N\setminus \varphi(T\times [0,\epsilon))$. Denote by $\nu$ the unit outward normal vector field of $\partial N_\epsilon$. For a small constant $\delta_\epsilon>0$, the map
\[   \gamma_\epsilon:(\partial N_\epsilon\setminus \partial_{rel}N)\times [-\delta_\epsilon,0]\to N_\epsilon , \quad (x,t)\longmapsto \exp(x,t\nu) \]
is well-defined and gives Fermi coordinates on one side of $\partial N_\epsilon$. Then by A. Song \cite{Song18}*{Section 2.2}, there exist smooth metrics $g_\epsilon$ on $N_\epsilon$ satisfying Lemma 4, Lemma 5 and Lemma 7 in \cite{Song18}, such that $(N_\epsilon, g_\epsilon)$ approaches $(\ms C(N),\hat g)$ in an appropriate sense. In particular, 
\begin{enumerate}[label=(\roman*)]
    \item $g_\epsilon = g$ in $N_\epsilon\setminus \gamma_\epsilon(\partial N_\epsilon\times [-\delta_\epsilon, 0])$;
    \item for $t \in [-\delta_\epsilon,0]$, the slices $\gamma_\epsilon(\partial N_{\epsilon}\times  \{t\})$ have non-zero mean curvature vector pointing towards $\partial N_\epsilon$ with respect to the new metric $g_\epsilon$;
    \item\label{item:gepsilon=g along surfaces} $\gamma_\epsilon^*(g_\epsilon)=\gamma_\epsilon^*(g)$ on $\partial N_\epsilon\times \{t\}$ for all $t\in [-\delta_\epsilon,0]$.
\end{enumerate}

Assume that $(N,\partial N,g)$ is isometrically embedded into a closed three-manifold $(M^3,\wti g)$. Then one can extend the metric $g_\epsilon$ to a metric $\wti g_\epsilon$ on $M$ so that $\wti g_\epsilon=g_\epsilon$ on $N_\epsilon$.

\subsection{A dichotomy in mean convex balls}

Let $N$ be a compact manifold with boundary $\partial N=\partial_{rel}N\cup T$. Suppose that $T$ is a stable minimal surface. 
Let $(\ms C(N),\hat g)$ be defined in Section \ref{subsec:cylindrical ends}. Recall that it can be approached by $(N_\epsilon, g_\epsilon)$ in an appropriate sense.

\begin{lemma}\label{lem:existence of fbmd}
Let $(N,\partial_{rel}N \cup T,g)$ be as above. Suppose that 
\begin{itemize}
    \item $N$ is diffeomorphic to a ball with finitely many disjoint balls removed;
    \item $\partial_{rel}N$ is mean convex and $T$ consists of stable minimal spheres with contracting neighborhoods;
    \item there is no embedded minimal sphere in $N\setminus \partial N$.
\end{itemize}
 Then there exists an embedded minimal disk in $N$ with free boundary on $\partial_{rel}N$.
\end{lemma}
\begin{proof}
Since there is no embedded minimal sphere in $N\setminus \partial N$, we know that $\partial_{rel}N$ is connected. Otherwise, let $\Gamma\subset \partial_{rel}N$ be a connected component and consider the isotopy area minimizer starting from $\Gamma$. Then the area minimizer $\Gamma'$ is a stable minimal surface and each connected component is a minimal sphere. %\textcolor{red}{(Argue that there exists at least one component distinct with $T$.)} 
Since $\partial_{rel}N$ is mean convex and $\Gamma'$ separates $\partial_{rel}N$, $\Gamma'$ must have a connected component in the interior of $N$, which contradicts the assumptions.

%Here $\partial_{rel}N$ is connected, but $T$ may not be connected.

\medskip
{\noindent\em Step I: Constructing balls to approximate $\ms C(N)$.}

\smallskip
Consider $\ms C(N)$ the non-compact $C^1$ manifold (endowed with the metric $\hat g$ which is defined separately on $N$ and ends) by adding cylindrical ends to $T$ as above. Then there exists $(N_\epsilon,g_\epsilon)$ converges to $(\ms C(N),\hat g)$ in appropriate sense. Now we suppose that $\partial N_\epsilon\setminus \partial_{rel}N$ consists of connected components $\Gamma_0,\Gamma_1,\cdots,\Gamma_k$.
Moreover, one can fill-in a Riemannian ball $B_i$ to each connected component $\Gamma_i$ such that $B_i$ admits a sweepout $\{S_i(t)\}_{0\leq t\leq 1}$ of spheres with 
\[
    \mc H^2(S_i(t))\leq \mc H^2(S_i(0))=\mc H^2(\Gamma_i).
\]
Denote by 
\[
    \wti N_\epsilon= N_\epsilon \cup (\cup_i B_i), 
\]
endowed with a metric $\wti g_\epsilon$ defined by combining the metrics on $B_i$ and $g_\epsilon$ on $N_\epsilon$. Note that $\wti N_\epsilon$ is a ball with smooth boundary $\partial_{rel} N$.

We will decompose $(\wti N_\epsilon, \wti g_\epsilon)$ into two connected components as follows. Recall that $g_\epsilon=g$ on $N_\epsilon\setminus \gamma_\epsilon ((\partial N_\epsilon\setminus \partial N)\times [-\delta_\epsilon,0])$. Denote by $\Gamma_i^\epsilon:=\gamma_\epsilon(\Gamma_i\times \{-\delta_\epsilon\})$.
Take disjoint smooth curves 
\[
\{\gamma_i:[0,1]\to N_\epsilon\setminus \gamma_\epsilon ((\partial N_\epsilon\setminus \partial N)\times (-\delta_\epsilon,0])\}\]
connecting $\Gamma_0^\epsilon$ and $\Gamma_i^\epsilon$ such that 
\begin{itemize}
    \item $\gamma_i$ meets $\Gamma_0^\epsilon$ and $\Gamma_i^\epsilon$ orthogonally;
    \item $\gamma_i(0)\in \Gamma_0^\epsilon$ and $\gamma_i(1)\in \Gamma_i^\epsilon$.
\end{itemize}
Set
\[  
    \wti B_i:= B_i\cup \gamma_\epsilon(\Gamma_i\times (-\delta_\epsilon,0]).
\]
Given $\eta>0$ small enough, denote by $C_\eta(\gamma_i)\subset N_\epsilon\setminus \wti B_i$ the $\eta$-tubular open neighborhood of $\gamma_i$ in $N_\epsilon\setminus \wti B_i$. Now let 
\[   
    \mk B_\epsilon:= \bigcup_{i=0}^k(\wti B_i\cup {C_\eta(\gamma_i)}),
\]
where $\gamma_0:=\emptyset$. Then $\mk B_\epsilon$ is homeomorphic to a ball and $\wti N_\epsilon\setminus \mk B_\epsilon$ is homeomorphic to a cylinder.

\medskip
{\noindent\em Step II: Constructing sweepouts of $\mk B_\epsilon$ and $\wti N_\epsilon\setminus \mk B_\epsilon$ with uniform area upper bounds.}

\smallskip

Note that $\wti g_\epsilon=g$ on $\wti N_\epsilon \setminus \mk B_\epsilon$. Hence $\wti N_\epsilon \setminus \mk B_\epsilon$ admits a sweepout of spheres with uniform area upper bound independent of $\epsilon$.

Then by the property of $\Gamma_i$ and $B_i$, we have that $\wti B_i$ is a Riemannian 3-ball admitting a sweepout $\{\wti S_i(t)\}_{t\in(-\delta_\epsilon,1]}$ by spheres defined as
\begin{gather*}
     \wti S_i(t)= \gamma_\epsilon (\Gamma_i\times \{t\}) \quad \forall t\in(-\delta_\epsilon,0];\quad    \wti S_i(t)=S_i(t) \quad \forall t\in (0,1].
\end{gather*}
Note that 
\[
   \mc H^2(\wti S_i(t))\leq \mc H^2(\Gamma_i\times \{-\delta_\epsilon\})< \mc H^2(\Gamma_i)+\frac{1}{k+1}. 
\]
Next we take a smooth curve $\ell_i:(-\delta_\epsilon,1] \to \mr{Clos}(\wti B_i)$ such that
\[    \ell_i(-\delta_\epsilon)=\gamma_i(1),  \quad \ell_i(1)=S_i(1); \quad \ell_i(t)\in \wti S_i(t) \quad \forall t\in(-\delta_\epsilon,1].\]
Moreover, one can also ensure that $\ell_i$ (transversally) intersects $\wti S_i(t)$ at exactly one point. 

Now we take a strictly decreasing smooth function $\xi:[-\delta_\epsilon, 1] \to [0,1]$ such that:
\[
 \xi(-\delta_\epsilon)=1, \quad \xi(1)=0.
\]
For each $t\in[-\delta_\epsilon,1]$, set
\[   
    \Phi(t):=\bigcup_{i=1}^k\Big[\Big(\wti S_i(t)\setminus B_{\eta\xi(t)}(\ell_i(t))\Big)\bigcup\Big(\cup_{s\leq t} \partial B_{\eta \xi(s)}(\ell_i(s))\Big)\Big]\bigcup \Big[   \partial \mk B_\epsilon\setminus \bigcup_{i=1}^k\partial \wti B_i \Big].
\]
Indeed, $\Phi(t)$ ($t\in [-\delta_\epsilon,1]$) is given by the connected sum of $\wti S_0(-\delta_\epsilon), \wti S_1(t),\cdots, \wti S_k(t)$ with small cylinders. 
Clearly,
\[
    \mc H^2(\Phi(t)) \leq \sum_{i=0}^k\mc H^2(\Gamma_i)+1,\quad \forall t\in [-\delta_\epsilon,1].
\]

Then for each $t\in[1, 2+\delta_\epsilon]$, set
\[
    \Phi(t):=\bigcup_{i=1}^k\Big[ B_{\eta\xi(2-t)}(\ell_i(2-t))\bigcup\Big(\cup_{s\leq t} \partial B_{\eta \xi(s)}(\ell_i(s))\Big)\Big]\bigcup \Big[   \partial \mk B_\epsilon\setminus \bigcup_{i=1}^k\partial \wti B_i \Big].
\]
Indeed, $\Phi(t)$ ($t\in [1,2+\delta_\epsilon]$) is given by adding small cylinders and disks to $\wti S_0(-\delta_\epsilon)$. Clearly, we have that 
\[
    \mc H^2(\Phi(t)) \leq \mc H^2(\Gamma_0)+1,\quad  \forall t\in [1,2+\delta_\epsilon].
\]

Then one can continue the construction of $\Phi$ on $(2+\delta_\epsilon,3]$ with $\Phi(3)=S_0(1)$ being a point, and 
\[
    \mc H^2(\Phi(t)) \leq \mc H^2(\Gamma_0)+1, \quad \forall t\in (2+\delta_\epsilon,3]. 
\]

Overall, we have constructed a sweepout $\{\Phi(t)\}_{t\in[-\delta_\epsilon,3]}$ of $\mk B_\epsilon$ with 
\[  \mc H^2(\Phi(t))\leq  \sum_{i=0}^k\mc H^2(\Gamma_i)+1  < \mc H^2(T)+2, \quad \forall t\in [-\delta_\epsilon,3].\]

\medskip
{\noindent\em Step III: Applying Simon--Smith min-max theory.}
\smallskip

So far, together with the sweepout of $\wti N_\epsilon\setminus \mk B_\epsilon$ and $\mk B_\epsilon$, we have constructed sweepouts on $\wti N_\epsilon$ by spheres. Then by opening necks, we can construct a sweepout of $\wti N_\epsilon$ by disks with uniform area upper bounds. Applying Simon--Smith min-max theory \cite{Li-CPAM}, there exists an embedded free boundary minimal surfaces $(D_\epsilon,\partial D_\epsilon)$ in $(N_\epsilon,\partial N_\epsilon, g_\epsilon)$ with genus 0, index $\leq 1$, and uniform area upper bounds (independent of $\epsilon$).
Note that $\gamma_\epsilon (\Gamma_i\times \{t\})$ has mean curvature vector pointing towards $\Gamma_i$. Hence $D_\epsilon$ has to intersect $\gamma_\epsilon (\Gamma_i\times \{t\})$ for all $t\in[-\delta_\epsilon,0]$ if $D_\epsilon$ intersects $\Gamma_i$. Then by the monotonicity formula and area upper bounds, we conclude that $D_\epsilon$ does not intersect $\Gamma_i$ for all $i=0,1,\cdots k$ when $\epsilon$ is small enough. Equivalently, $(D_\epsilon ,\partial D_\epsilon)\subset (N_\epsilon,\partial_{rel}N,g_\epsilon)$ for all small $\epsilon$. By the topological control of the min-max minimal surfaces, we obtain that $D_\epsilon$ is a disk. Finally by the compactness of free boundary minimal surfaces and the argument of Song \cite{Song18}, $D_\epsilon$ subsequentially converges to an embedded  free boundary minimal disk or sphere in $(N\setminus T,g)$ as $\epsilon\to 0$. Recall that there is no embedded minimal sphere in $N\setminus T$. Hence the limit is the desired surface. This finishes the proof of Lemma \ref{lem:existence of fbmd}.
\end{proof}

\begin{theorem}\label{thm:always fbmd}
Suppose that $(M,g)$ is a Riemannian 3-ball with strictly mean convex boundary. Then there exists an embedded free boundary minimal disk. Moreover,
\begin{enumerate}
    \item either there is no closed immersed minimal surface; or 
    \item there exists an embedded stable minimal sphere.
\end{enumerate}
\end{theorem}
\begin{proof}
Consider the mean curvature flow starting at $\partial M$. Then by White \cite{Whi00}*{Theorem 11.1}, the weak solution will either be extinct or converge to embedded minimal spheres. If it becomes extinct in finite time, then it does not contain any closed immersed minimal surfaces. 

We now suppose that the weak mean curvature flow converges to embedded minimal spheres. Denote by $\Sigma$ the limit surface which is possibly disconnected. Denote by $N$ the connected component of $M\setminus \Sigma$ which contains $\partial M$. Consider the metric completion of $N$ w.r.t. $g$, still denoted by $(N, g)$. Note that $\Sigma$ has a {\em contracting neighborhood} $U$ in $N$. Clearly, there is no embedded minimal sphere in $N\setminus \partial N$. Then by Lemma \ref{lem:existence of fbmd}, there exists an embedded minimal disk in $N$ with free boundary on $\partial M$. This completes the proof of Theorem \ref{thm:always fbmd}.
\end{proof}

\subsection{A dichotomy in mean convex balls with bumpy metrics}
\begin{theorem}[cf. Theorem \ref{thm: main}]\label{thm:bumpy assumption}
Let $(M,g)$ be a Riemannian three-ball with strictly mean convex boundary such that there are no degenerate-and-stable minimal spheres or free boundary minimal disks. Then
\begin{enumerate}
    \item either there are three embedded minimal spheres; or 
    \item there are three embedded free boundary minimal disks.
\end{enumerate}
\end{theorem}
\begin{proof}[Proof of Theorem \ref{thm:bumpy assumption}]
By Theorem \ref{thm:always fbmd}, either there is no closed immersed minimal surface in $(M,g)$, or there exists an embedded stable minimal sphere.

Suppose that there exists an embedded stable minimal surface $\Sigma$. By assumption, $\Sigma$ is strictly stable. Denote by $N$ the metric completion of the ball bounded by $\Sigma$ in $M$. Then $(N,g)$ is a Riemannian ball whose boundary is a strictly stable minimal sphere. By Wang-Zhou \cite{Wang-Zhou23}, there are at least two embedded unstable minimal spheres in $N\setminus \partial N$. Hence there are three embedded minimal spheres in $(M,g)$.

Now we suppose that $(M,g)$ does not contain immersed closed minimal surfaces. We divide the proof into two cases.

\medskip 
{\noindent\em Case I: $(M,g)$ does not contain stable free boundary minimal disks.}

\smallskip
Then any two embedded free boundary minimal disks must intersect; otherwise, minimizing area in the region bounded between them would produce a stable free boundary minimal disk, contradicting our assumption. By Simon--Smith min-max theory and the multiplicity one theorem, there exists an embedded free boundary minimal disk $\Sigma_i$ such that 
\[
   \mf L(\ms P_i)=\Area(\Sigma_i),
\]
where $\mf L(\ms P_i)$ is defined in Section \ref{subsec:three sweepouts}.
Without loss of generality, we also assume that there are finitely many embedded free boundary minimal disks. Then by Lemma \ref{lem:LS theory}, 
\[
    \mf L(\ms P_1)<\mf L(\ms P_2)<\mf L(\ms P_3).
\]
Hence $\Sigma_1,\Sigma_2,\Sigma_3$ are different embedded free boundary minimal disks.

\medskip
{\noindent\em Case II: $(M,g)$ contains stable free boundary minimal disks.}

\smallskip
Denote by $\Gamma$ the stable free boundary minimal disk. By assumption, $\Gamma$ is strictly stable. Note that $\Gamma$ separates $M$ into two connected components. By Theorem \ref{thm:free boundary local min-max}, there exists an unstable free boundary minimal disk in $(M,\partial M,g)$ in each connected component. This finishes the proof.
\end{proof}

\begin{remark}
In Case I, we have proved that if $(M,g)$ does not contain any embedded stable free boundary minimal disks or minimal spheres, then $(M,g)$ admits at least three distinct embedded free boundary minimal disks.
\end{remark}

\begin{comment}
\begin{theorem}
Let $(M,g)$ be a Riemannian 3-ball with strictly mean convex boundary. Suppose that $(M,g)$ does not contain any degenerate-and-stable free boundary minimal disks. Then 
\begin{enumerate}
    \item either there are two embedded minimal spheres and one free boundary minimal disks; or
    \item there are three embedded free boundary minimal disks.
\end{enumerate}
\end{theorem}
\begin{proof}

\end{proof}
\end{comment}

\appendix

\section{Minimal disks in manifold with barriers}\label{appen:barries}

In this section, we assume that  $N:=\{(x,y,z)\in \mb R^3;x^2+y^2+z^2\leq 1; z\geq 0\}$ is a half-ball whose boundary has two smooth parts
\[ \partial_{rel}N:=\{(x,y,z)\in \mb R^3;x^2+y^2+z^2=1; z\geq 0\}; \quad T:=\{x^2+y^2\leq 1; z=0\}.\]
Let $g$ be a smooth metric on $B^3$. We restrict the metric $g$ to $N$.

Now we suppose that $T$ is a strictly stable free boundary minimal surface in $(N,\partial_{rel}N)$. 
Then by \cite{Wang_21_2}*{Section 3.1}, there exist $\mu>0$, and a diffeomorphism $\varphi: T\times [0,\mu)\to N$ such that 
\begin{itemize}
    \item $\varphi$ is a diffeomorphism to its image;
    \item $\varphi(x,0)=x$ for all $x\in T$, and $\varphi(T,s)$ is a free boundary surface in $(M,\partial_{rel}M)$ with mean curvature vector pointing towards $T$;
    \item for all $t\in [0,\mu)$,
    \[    \varphi_*\left(\frac{\partial}{\partial t}\right)= \frac{\nabla t}{|\nabla t|^2}.\]
\end{itemize}
 Clearly, there exists a positive smooth function $f$ on $\varphi(T\times [0,\mu])$ so that the metric $g$ can be written as
\begin{equation}\label{eq:decompose metric}
g=g_t(q)\oplus (f(q)\mr dt)^2, \ \ \ \forall q\in \varphi(T\times\{t\}). 
\end{equation}
Here $g_t=g\llcorner \varphi(T\times \{t\})$ is the restricted metric and it can be extend to define a 2-form over $TN$.

Note that $\varphi:T\times \{0\}\rightarrow T$ is the canonical identifying map. Define the following non-compact manifold with cylindrical ends:
\[\ms C(N):=N\cup_{\varphi}(T\times [0,+\infty)).\]
We endow it with the metric $h$ such that $h=g$ on $N$ and 
\begin{equation}\label{eq:def of h}
h=g\llcorner T\oplus (f_0dt)^2
\end{equation}
on $T\times [0,+\infty)$. Here $g\llcorner T$ is the restriction of $g$ to the tangent bundle of $T$ and 
\[f_0(x,t)=f(\varphi(x,0))=\phi(\varphi(x));\]
see \eqref{eq:decompose metric} for the definition of $f$. We remark that under the metric $h$, each slice $T\times\{t\}$ is totally geodesic.

 Now for any $\epsilon<\mu$, define on $N$ the following metric $h_\epsilon$:
\[
h_\epsilon(q):=
\left\{
\begin{array}{ll}
g_t(q)\oplus(\vartheta_\epsilon(t)f(q)dt)^2 &\text{\ for \ } q\in  \varphi(T\times[0, \epsilon])\\
g(q) &\text{\ for \ } q\in N\setminus F(T\times [0, \epsilon]).
\end{array} 
\right.
\]
Here $\vartheta_\epsilon$ is chosen to be a smooth function on $[0,\epsilon]$ so that
\begin{itemize}
\item $1\leq \vartheta_\epsilon$ and $\frac{\partial}{\partial t}\vartheta_\epsilon\leq 0$;
\item $\vartheta_\epsilon\equiv1$ in a neighborhood of $\epsilon$;
\item $\lim_{\epsilon\rightarrow0}\int_{\epsilon/2}^\epsilon\vartheta_{\epsilon}=+\infty$;
%\item $\lim_{\epsilon\rightarrow0}\int_{z_\epsilon}^\epsilon\vartheta_\epsilon=0$.
\end{itemize}

By \cite{Wang_21_2}*{Lemma 3.2}, with the differential structure associated with $h$ on $N$ and $T\times (0,+\infty)$, $\ms C(N)$ is a $C^1$ manifold with boundary. Moreover, the metric is Lipschitz continuous. 
For simplicity, denote by
\[
    N_\epsilon:=  N\setminus \varphi(T\times [0,\epsilon/2)),
\] 
and 
\[
    T_\epsilon:= \varphi(T\times \{\epsilon/2\}), \quad \partial_{rel}N_\epsilon :=\partial N\setminus \varphi (\partial T\times [0,\epsilon/2)).
\]
Clearly, $(N_\epsilon,\partial_{rel}N_\epsilon \cup T_\epsilon,g_\epsilon)$ is a compact manifold with boundary $\partial_{rel}N_\epsilon \cup T_\epsilon$ such that $T_\epsilon$ meets $\partial_{rel}N_\epsilon$ orthogonally.
Then by A. Song \cite{Song18}*{Section 2.2} (see \cite{Wang_21_2}*{Section 3.1} for free boundary versions), $(N_\epsilon, h_\epsilon)$ approaches $(\ms C(N),h)$ in appropriate sense.
In particular, 
\begin{enumerate}[label=(\roman*)]
    \item $h_\epsilon = g$ in $N\setminus \varphi(T\times [0,\mu))$;
    \item for $t \in (0,\mu)$, the slices $\varphi(T\times  \{t\})$ have non-zero mean curvature vector pointing towards $T$ with respect to the new metric $h_\epsilon$.
\end{enumerate}

\begin{theorem}\label{thm:free boundary local min-max}
    Let $(N,\partial_{rel}N \cup  T,g)$ be a compact manifold with boundary $\partial_{rel}N\cup T$ such that $T$ meets $\partial_{rel}N$ orthogonally. Suppose that $T$ is a stable free boundary minimal disk in $(N,\partial_{rel}N)$ with contracting (one-sided) neighborhood and $\partial_{rel}N$ is mean convex. Then there exists an embedded free boundary minimal disk $(D,\partial D)\subset (N\setminus T,\partial_{rel}N)$. 
\end{theorem}
\begin{proof}
Consider $\ms C(N)$ the non-compact $C^1$ manifold (endowed with the metric $h$ which is defined separately on $N$ and ends as above) by adding cylindrical ends to $T$ as above. Then there exists $(N_\epsilon,h_\epsilon)$ converges to $(\ms C(N),h)$ in appropriate sense.

Moreover, one can fill-in a Riemannian half-ball $B^+$ by identifying the half-sphere with $T_\epsilon$ such that $B^+$ admits a sweepout $\{S_i(t)\}_{0\leq t\leq 1}$ of free boundary disks with 
\[
    \mc H^2(S_i(t))\leq \mc H^2(S_i(0))=\mc H^2(T_\epsilon).
\]
Denote by 
\[
    \wti N_\epsilon= N_\epsilon \cup B^+, 
\]
endowed with a metric $\wti h_\epsilon$ defined by combining the metrics on $B^+$ and $h_\epsilon$ on $N_\epsilon$. Note that $\wti N_\epsilon$ is a ball with smooth boundary and $(\wti N_\epsilon,\wti h_\epsilon)$ also converges to $(\ms C(N),g)$ in appropriate sense.

Denote by 
\[ 
    \wti B:= B\cup \varphi(T\times [\epsilon/2,\epsilon))\subset \wti N_\epsilon.
\]
The sweepout of $B^+$ and the foliation of $\varphi (T\times [0,\epsilon)$ give a sweepout of $\wti B$ with uniform area upper bounds (independent of $\epsilon$). Note that $\wti h_\epsilon= g$ on $\wti N_\epsilon\setminus \wti B$. Then there exists a sweepout of $\wti N_\epsilon\setminus \wti B$ from $\varphi(T\times \{\epsilon\})$. These facts combined imply that there exists a sweepout $\{\Phi(t)\}_{t\in[0,1]}$ of $(\wti N_\epsilon,\wti h_\epsilon)$ by disks (and points as degenerate disks) with uniform area upper bounds. Recall that $\wti N_\epsilon$ is a ball. Then we can double it to obtain a sphere $S$ and construct a sweepout $\{ \wti \Phi(t) \}_{t\in[0,1]}$ of $S$ by spheres such that 
\[    \wti\Phi(t)\cap \wti N_\epsilon = \Phi(t).\]
Denote by $\Pi$ the homotopy class that contains $\wti \Phi$.
Define
\[
W:= \inf_{\Psi\in \Pi} \sup_{t\in(0,1)} \mc H^2(\Psi(t)\cap N_\epsilon).
\]
Then by the work of Li \cite{Li-CPAM}, there exists an integral varifold $V_\epsilon$ such that $W=\|V_\epsilon\|(N_\epsilon)$ and $V_\epsilon$ is a smooth free boundary minimal surface outside $\overline {B^+}$. Moreover, 

\begin{claim}
$V_\epsilon$ does not intersect $\partial T_\epsilon$ for all sufficiently small $\epsilon$.
\end{claim}
\begin{proof}   
Suppose not, by \cite{Wang_21_2}*{Lemma 2.13}, we can prove that the support of $V$ will intersect $U \setminus \varphi(T\times \{\epsilon/2\})$ for any relative open set $U\subset N_\epsilon$ that contains $\partial T_\epsilon$. By the maximum principle, $V_\epsilon$ also has to intersect $F (T \times\{t\})$. Note that $\mf M(V_\epsilon) $ is uniformly bounded from above. This contradicts the monotonicity formula \citelist{\cite{GLZ16}*{Theorem 3.4} \cite{Si}*{\S 17.6}}.
\end{proof}

Hence the regularity theorem gives that $V_\epsilon$ is induced by smooth free boundary minimal surfaces. Then by the same argument as in \cite{Wang_21_2}*{Theorem 3.10}, one can prove that $V_\epsilon$ subsequently converges to a smooth free minimal surface in $(N\setminus T,g)$. By the compactness of free boundary minimal hypersurfaces, we then obtain that $V_\epsilon$ is a smooth minimal surface in $N\setminus \varphi(T\times [0,\epsilon])$. By the topological control, we conclude that $V_\epsilon$ is a free boundary minimal disk in $(N\setminus T,\partial_{rel}N)$.
\end{proof}

\section{Proof of the Lusternik–Schnirelmann inequality}\label{appen:LS inequality}

Consider the fibration
\[
  \mr{Emb}(D^2,B^3\mr{rel}\,S^1)\to \mr{Emb}(D^2,B^3)\to \mr{Emb}(S^1,S^2).
\]
By Smale's conjecture proved by Hatcher \cite{Hat}*{Statement (5) in Appendix}, $\mr{Emb}(D^2,B^3\mr{rel}S^1)$ is contractible. By Smale, $\mr{Emb}(S^1,S^2)\simeq \mb{RP}^2$. Therefore, one concludes that 
$\mr{Emb}(D^2,B^3)$ is homotopic to $\mb{RP}^2$.

\begin{proof}[Proof of Lemma \ref{lem:LS theory}]
We prove the last inequality and the others are similar. 
By the definition of $\mf L(\ms P_3)$, there exists a sequence of $\{\Phi_i:(X_i,Z_i)\to (\oli{\ms X},\ms Y)\}\subset \ms P_3$ such that 
\[   
\mf L(\ms P_3)=\lim_{i\to\infty}\max_{x\in X_i} \Area(\Phi_i(x)).
\]
Denote by $\mc S$ the collection of integral varifolds, with mass equal to $\mf L(\ms P_3)$, whose supports are embedded minimal spheres or embedded free boundary minimal disks.
Given $\eta_1>0$, define 
\begin{gather*}
    Y_i:=\{x\in X_i:\,\mf F(|\Phi_i(x)|,\mc S)\geq \eta_1\};\quad      K_i:=\oli{X_i\setminus Y_i}.
\end{gather*}
Note that $K_i\subset\interior(X_i)$ for small enough $\eta_1$. In particular, $K_i\cap Z_i=\emptyset$. Denote by $\iota_1:K_i\to X_i$ and $\iota_2:Y_i\to X_i$ the two natural inclusion maps. Since $\Phi_i\in \ms P_3$, there exists $\bar\lambda\in H^1(\oli{\ms X},\ms Y;\mb Z_2)$ such that $[\Phi_i^*(\bar\lambda)]^3\neq 0\in H^3(X_i,Z_i;\mb Z_2)$. Observe that the following diagram
\[
\begin{tikzcd}
    & H^1(\oli{\ms X},\ms Y;\mb Z_2)\arrow[r,"\wti j^*"] \arrow[d,"\Phi_i^*"]
        & H^1(\oli{\ms X};\mb Z_2)\arrow[d,"\wti{ \Phi_i^*}"]\\
    H^1(X_i,K_i\cup Z_i;\mb Z_2)\arrow[r,"j_1^*"] 
        & H^1(X_i,Z_i;\mb Z_2)\arrow[r,"\iota^*_1"]
            & H^1(K_i\cup Z_i,Z_i;\mb Z_2)\rlap{$\simeq H^1(K_i;\mb Z_2)$}
\end{tikzcd}\hspace{2.4cm}
\]
is commutative. One can take 
$\eta_1$ small enough so that
\[
\wti {\Phi_i^*}(H^1(\oli{\ms X};\mb Z_2))=\{0\}.
\]
To see this, consider the chain of maps
$
\begin{tikzcd}
 K_i \arrow[r, "\iota_1"]
& \interior(X_i) \arrow[r, "\Phi_i"]
& \ms X \arrow[r, "\iota"]
&\mc Z_2(M,\partial M; \mb Z_2),
\end{tikzcd}
$
where $\iota: \ms X\to \mc Z_2(M,\partial M; \mb Z_2)$ is the natural inclusion map. By the argument in \cite{MN17}*{Section 6} and our choice of $K_i$, we know that $(\iota\circ\Phi_i\circ\iota_1)^*: H^1(\mc Z_2(M,\partial M; \mb Z_2); \mb Z_2)\to H^1(K_i; \mb Z_2)$ is trivial for small enough $\eta_1$. By the fact that $\ms X$ is homotopic to $\mb{RP}^2$, we also know that $\iota^*: H^1(\mc Z_2(M,\partial M; \mb Z_2);\mb Z_2) \to H^1(\ms X; \mb Z_2)$ is an isomorphism. Together, these imply that $(\Phi_i\circ\iota_1)^*: H^1(\ms X;\mb Z_2)\to H^1(K_i;\mb Z_2)$ is trivial, and this implies $\wti {\Phi_i^*}(H^1(\oli{\ms X};\mb Z_2))=\{0\}$.

It then follows that 
\begin{equation*}
    \iota_1^*\circ \Phi_i^*(\bar\lambda)=0.
\end{equation*}
Note that the sequence in the second line is exact. Hence there exists $\alpha\in H^1(X_i,K_i\cup Z_i;\mb Z_2)$ such that $j_1^*(\alpha)=\Phi_i^*(\bar\lambda)$. On the other hand, the following sequence 
\[   
\begin{tikzcd}
H^2(X_i,Y_i;\mb Z_2)\arrow[r,"j_2^*"] 
        &H^2(X_i,Z_i;\mb Z_2)\arrow[r,"\iota^*_2"]
            &H^2(Y_i,Z_i;\mb Z_2)
\end{tikzcd}
\]
is also exact.

Note that $Y_i\cup K_i=X_i$. Consider the diagram
\[
\begin{tikzcd}
    H^1(X_i,K_i \cup Z_i;\mb Z_2) \arrow[d,"j_1^*"] \arrow[r, phantom, "\times"] & H^2(X_i,Y_i;\mb Z_2) \arrow[r,"\smile"] \arrow[d,"j_2^*"] & H^3(X_i,X_i;\mb Z_2) \arrow[d, "j_3^*"]\\
    H^1(X_i,Z_i;\mb Z_2) \arrow[r, phantom, "\times"] & H^2(X_i,Z_i;\mb Z_2) \arrow[r,"\smile"] & H^3(X_i,Z_i;\mb Z_2).
\end{tikzcd}
\]
By the naturality, then we have
\[ j_1^*\big(H^1(X_i,K_i\cup Z_i;\mb Z_2)\big)\smile j_2^*\big(H^2(X_i,Y_i;\mb Z_2)\big)\subset j_3^*(H^3(X_i,X_i;\mb Z_2))=\{0\}.\] 
Together with the fact that 
\[ j_1^*(\alpha)\smile [\Phi_i^*(\bar\lambda)]^2=\Phi_i^*(\bar\lambda)\smile [\Phi_i^*(\bar\lambda)]^2\neq 0,\]
we then conclude that 
\[ [\Phi_i^*(\bar\lambda)]^2\notin \mr{Im} j_2^*=\ker \iota_2^*;\]
that is, $\iota_2^*[\Phi_i^*(\bar\lambda)]^2\neq 0\in H^3(Y_i,Z_i;\mb Z_2)$. Hence we have that $\{\Phi_i:(Y_i,Z_i)\to (\oli{\ms X},\ms Y)\}\subset \ms P_2$. By the definition of $Y_i$, then after the tightening process, we can deform $\Phi_i|_{Y_i}$ to $\Phi'_i$ such that for all sufficiently large $i$,
\begin{equation*}
\sup_{x\in Y_i}\Area(\Phi_i(x))< \mf L(\ms P_3)-\epsilon,
\end{equation*}
which implies that
%finally have that 
\[ \mf L(\ms P_2)<\mf L(\ms P_3).\]
This completes the proof.
\end{proof}

\bibliographystyle{amsalpha}
\bibliography{three_disks}
\end{document}